\documentclass[12pt]{amsart}
\usepackage{amssymb}
\usepackage{mathtools}
\usepackage{txfonts}
\usepackage{eucal}
\usepackage{extarrows}
\usepackage{wasysym}
\usepackage[all]{xy}

\usepackage{tikz}

\usepackage{graphicx}
\usepackage{float}

\usepackage{xspace}

\usepackage[a4paper,body={16.3cm,22.8cm},centering]{geometry}
\usepackage[colorlinks,final,backref=page,hyperindex,pdftex]{hyperref}
\usepackage{shuffle}

\font \sevenrm=cmr7
\font \fiverm=cmr5
\newcommand{\nc}{\newcommand}
\nc\smsc{0.8}

\def\oprec{\!\!\joinrel{\ocircle\hskip -12.5pt \prec}\,}
\def\soprec{\,\joinrel{\ocircle\hskip -6.7pt \prec}\,}

\def\diagramme #1{\vskip 4mm \centerline {#1} \vskip 4mm}

\def \restr#1{\mathstrut_{\textstyle |}\raise-6pt\hbox{$\scriptstyle #1$}}
\def \srestr#1{\mathstrut_{\scriptstyle |}\hbox to
-1.5pt{}\raise-4pt\hbox{$\hskip 1pt\scriptscriptstyle #1$}}
\nc{\mop}[1]{\mathop{\hbox {\rm #1} }\nolimits}
\nc{\gmop}[1]{\mathop{\hbox {\bf #1} }\nolimits}

\nc{\smop}[1]{\mathop{\hbox {\sevenrm #1} }\nolimits}
\nc{\ssmop}[1]{\mathop{\hbox {\fiverm #1} }\nolimits}
\nc{\mopl}[1]{\mathop{\hbox {\rm #1} }\limits}
\def\dbar{d\hskip-3pt \raise 4pt\hbox{-}}
\nc{\smopl}[1]{\mathop{\hbox {\sevenrm #1} }\limits}
\nc{\ssmopl}[1]{\mathop{\hbox {\fiverm #1} }\limits}

\newcommand{\delete}[1]{}

\nc{\mlabel}[1]{\label{#1}}  
\nc{\mcite}[1]{\cite{#1}}  
\nc{\mref}[1]{\ref{#1}}  
\nc{\mbibitem}[1]{\bibitem{#1}} 

\delete{
\nc{\mlabel}[1]{\label{#1}  
{\hfill \hspace{1cm}{\small\tt{{\ }\hfill(#1)}}}}
\nc{\mcite}[1]{\cite{#1}{\small{\tt{{\ }(#1)}}}}  
\nc{\mref}[1]{\ref{#1}{{\tt{{\ }(#1)}}}}  
\nc{\mbibitem}[1]{\bibitem[\bf #1]{#1}} 
}

\newtheorem{theorem}{Theorem}[section]
\newtheorem{definition}{Definition}[section]

\newtheorem{proposition}{Proposition}[section]
\newtheorem{lemma}{Lemma}[section]
\newtheorem{remark}{Remark}[section]

\numberwithin{equation}{section}
\newcommand\alphlist{a,b,c,d,e,f,g,h,i,j,k,l,m,n,o,p,q,r,s,t,u,v,w,x,y,z}
\newcommand\Alphlist{A,B,C,D,E,F,G,H,I,J,K,L,M,N,O,P,Q,R,S,T,U,V,W,X,Y,Z}
\newcommand\getcmds[3]{\expandafter\newcommand\csname #2#1\endcsname{#3{#1}}}
\makeatletter
\@for\x:=\alphlist\do{\expandafter\getcmds\expandafter{\x}{frak}{\mathfrak}}
\@for\x:=\Alphlist\do{\expandafter\getcmds\expandafter{\x}{frak}{\mathfrak}}
\makeatother

\nc{\bfk}{{\bf k}}
\nc{\sha}{\shuffle}

\nc{\id}{\mathrm{id}}
\nc{\Id}{\mathrm{Id}}
\nc{\lbar}[1]{\overline{#1}}
\nc{\ot}{\otimes}
\nc{\dep}{\mathrm{dep}}
\nc{\ver}{\mathrm{ver}}

\nc{\tred}[1]{\textcolor{red}{#1}} \nc{\tgreen}[1]{\textcolor{green}{#1}}
\nc{\tblue}[1]{\textcolor{blue}{#1}} \nc{\tpurple}[1]{\textcolor{purple}{#1}}
\nc{\tcyan}[1]{\textcolor{cyan}{#1}} 
\nc{\tblk}[1]{\textcolor{black}{#1}}

\nc{\li}[1]{\tpurple{\underline{Li:}#1 }}
\nc{\liadd}[1]{\tpurple{#1}}
\nc{\xing}[1]{\tblue{\underline{Xing:}#1 }}
\nc{\yuan}[1]{\tred{\underline{Yuan:}#1 }}
\nc{\markus}[1]{\tred{\underline{Markus:} #1}}
\nc{\dominique}[1]{\tpurple{\underline{Dominique: }#1 }}
\long\def\ignore#1{}

\usetikzlibrary{calc,shapes.geometric}
\tikzset{
baseon/.style={baseline={($(#1)+(0,-0.58ex)$)}},
baseon/.default=current bounding box.center,
every picture/.style=baseon,
lst/.style={},
dst/.style={circle,inner sep=1pt,outer sep=0pt,fill,draw,dst2},
dst2/.style={fill=white},
ddst/.style={diamond,draw,inner sep=1pt},
eest/.style={ellipse,draw,inner sep=1pt,minimum size=2ex},
}

\def\zzz#1`#2...#3`#4...#5`#6@{%
--++(#1)
node[dst,label={#5:$#6$},name=#2]{}
node[midway,auto,#3]{$#4$}
}
\def\ddd#1`#2`#3@{+(#1)node[ddst,name=#2]{$#3$}}
\def\eee#1`#2`#3@{+(#1)node[eest,name=#2]{$#3$}}
\def\xxx#1`#2@{node[midway,auto,inner sep=1pt,#1]{$#2$}}
\def\pp#1`#2`#3@{node[dst,label={#2:$#3$},pos=#1]{}}
\def\oo#1`#2`#3@{\path (o) node[dst,label={#2:$#3$},name=o,#1]{};}
\def\eoo#1`#2@{\node[eest,name=o,#1] at (o) {$#2$};}

\newif\ifshowjdq
\showjdqtrue
\newcommand\setXXclip[3]{%
\def\XXheight{#1}\def\XXdepth{#2}\def\XXwidth{#3}}
\setXXclip{1}{-0.5}{1.1}

\makeatletter
\newcommand\simra{\mathrel{\mathpalette\@verra\sim}}
\def\@verra#1#2{\lower.5\p@\vbox{\lineskiplimit\maxdimen \lineskip-.5\p@
\ialign{$\m@th#1\hfil##\hfil$\crcr#2\crcr\rightarrow\crcr}}}
\makeatother

\nc{\dnx}{\Delta_n A} \nc{\dx}{\Delta A} \nc{\dgp}{{\rm deg_{P}}}
\nc{\dgt}{{\rm deg_{T}}} \nc{\dg}{{\rm deg}} \nc{\ida}{ID($A$)} \nc{\tu}{\tilde{u}} \nc{\tv}{\tilde{v}}
\nc{\nr}{\calr_n} \nc{\nz}{\calz_n} \nc{\fun}{\cala_{n,d}}
 \nc{\fbase}{\calb} \nc{\LF}{\mathrm{RF}} \nc{\FFA}{\mathrm{LF}} \nc{\irr}{\mathrm{Irr}}
 \nc{\result}{\bfk\mathrm{Irr}(S_n)}  \nc{\I}{I_{\mathrm{ID},n}^0}
 \nc{\nrs}{\calr_n^\star} \nc{\ii}{\mathrm{I}} \nc{\iii}{\mathrm{II}}
\nc{\intl}{{\rm int}}\nc{\ws}[1]{{#1}}\nc{\deleted}[1]{\delete{#1}}\nc{\plas}{placements\xspace}

\nc{\bim}[1]{#1}  \nc{\shaop}{\sha_{\Omega}^{+}}  \nc{\shao}{\sha_{\Omega}}
\nc{\bbim}[2]{#1 #2} \nc{\bbbim}[2]{#1,\, #2} \nc{\RBF}{{\rm RBF}}
\nc{\frb}{F_{\RB}} \nc{\shaf}{\ssha_{\tiny{\Omega}}} \nc{\sham}{\diamond_{\tiny{\Omega}}}
\nc{\lf}{\lfloor} \nc{\rf}{\rfloor} \nc{\shan}{\ssha_{\lambda}}
\nc{\rlex}{{\rm {lex}}} \nc{\bb}{\Box} \nc{\ra}{\rightarrow}
\nc{\e}{{\rm {e}}}
\nc{\DDF}{\mathrm{DD}(X,\,\Omega)}\nc{\DTF}{\mathrm{DT}(X,\,\Omega)} \nc{\DT}{\mathrm{DT}'(\Omega,\,V)}
\nc{\bra}{\mathrm{bra}} \nc{\bre}{\mathrm{bre}}
\nc{\dec}{\mathrm{dec}} \nc{\diamondw}{\diamond_{w}}
\nc{\type}{\mathrm{type}}

\nc\caF[1]{\cal{F}_{#1}(X,\,\Omega)}
\nc\calt{\cal{T}(X,\,\Omega)} \nc\caltn{\cal{T}_n(X,\,\Omega)}
\nc\calta{\cal{T}_0(X,\,\Omega)}
\nc\caltb{\cal{T}_1(X,\,\Omega)}
\nc\caltc{\cal{T}_2(X,\,\Omega)}
\nc\caltd{\cal{T}_3(X,\,\Omega)}
\nc\caltm{\cal{T}_m(X,\,\Omega)}
\nc\caltx{\cal{T}(X)}
\nc\calf{\cal{F}(X,\,\Omega)}
\nc\fram{\frak{M}(\Omega,\, X)}
\nc\shaw{\sha^{NC}_w(\Omega,\, X)}
\nc\dw{\diamond_w} \nc\dl{\diamond_\ell}
\nc\shal{\sha^{NC}_\ell(X,\, \Omega)} \nc\shav{\sha^{NC}_w(\Omega,\, V)} \nc\shat{\sha^{NC,1}_w(\Omega,\, T^{+}(V))}
\nc{\cfo}{\cal{F}(X,\,\Omega)}
\nc{\sh}{\rm{Sh}}
\nc{\lar}{\varinjlim}
\nc\XO{(X,\,\Omega)}
\def\cxo#1#2;{\cal{#1}#2\XO}
\nc\lrf[2]{B_{#2}^+(#1)}
\nc{\fd}{\mathrm{\text{typed angularly decorated planar rooted trees}}}
\nc{\rb}{\mathrm{RBFWs}} \nc{\dfw}{\mathrm{DFW{(X)}}} \nc{\tfw}{\mathrm{TFW{(X)}}}
\nc{\tfv}{\mathrm{TFW{(V)}}}

\def\Ve#1,#2,#3;{\vee_{#1,\,(#2,\,#3)}}
\def\bigv#1;#2;#3;{\bigvee\nolimits_{#1}^{#2;\,#3}}
\nc\rjt[2]{\mathrel{\mathop{\longrightarrow}\limits^{#1\hfill}_{\hfill#2}}}
\nc{\pl}{\cal{PLF}}
\nc{\tr}{\cal{RTF}}
\nc{\im}{\mathrm{Im}}
\nc{\ff}{\cal{F}_\Omega}
\nc{\tm}{T_\Omega}
\nc{\calp}{\cal{P}}
\makeatletter
\nc\dd{\@ifnextchar'{\ddA}{\ddB}}
\def\ddA'#1;{\rhd'_{#1\,}}
\def\ddB#1;{\rhd_{#1\,}}
\nc{\pbt}{\mathrm{PBT}}
\nc{\ad}{\mathrm{ad}}

\makeatother
\begin{document}

\title[Doubling bialgebras]{Doubling bialgebras of finite topologies}
\thispagestyle{empty}
\author{Mohamed Ayadi}
\address{Laboratoire de Math\'ematiques Blaise Pascal,
CNRS--Universit\'e Clermont-Auvergne,
3 place Vasar\'ely, CS 60026,
F63178 Aubi\`ere, France, and University of Sfax, Faculty of Sciences of Sfax,
LAMHA, route de Soukra,
3038 Sfax, Tunisia.}
\email{mohamed.ayadi@etu.uca.fr}
\author{Dominique Manchon}
\address{Laboratoire de Math\'ematiques Blaise Pascal,
CNRS--Universit\'e Clermont-Auvergne,
3 place Vasar\'ely, CS 60026,
F63178 Aubi\`ere, France}
\email{Dominique.Manchon@uca.fr}

		\tikzset{
			stdNode/.style={rounded corners, draw, align=right},
			greenRed/.style={stdNode, top color=green, bottom color=red},
			blueRed/.style={stdNode, top color=blue, bottom color=red}
		}
	
	\begin{abstract}
		The species of finite topological spaces admits two graded bimonoid structures, recently defined by F. Fauvet, L. Foissy, and the second author. In this article, we define a doubling of this species in two different ways. We build a bimonoid structure on each of these species and describe a cointeraction between them. We also investigate two related associative products obtained by dualisation.
	\end{abstract}
	
\keywords{Finite topological spaces, Species, Bimonoids, Bialgebras, Hopf algebras, Comodules}
\subjclass[2010]{16T05, 16T10, 16T15. 16T30, 06A11}
\maketitle
\tableofcontents
	\section{Introduction and preliminaries }
	The present article considers all finite topological spaces at once, and aims at investigating how they organise themselves into a rich algebraic structure, along the lines opened by L.~Foissy, C.~Malvenuto and F.~Patras in \cite{acg11,acg12}. It is a follow-up of \cite{acg10} by F.~Fauvet, L.~Foissy and the second author, in the context of a doubling procedure first developed by M.~Belhaj Mohamed for specified Feynman graphs \cite{acg3a, acg3b}.\\
	
	The general idea of the doubling procedure can be outlined as follows: from a connected graded Hopf algebra $\mathcal H$ linearly generated by a basis $\mathcal B$ of combinatorial objects, on can generate a bialgebra $\mathcal D$ from ordered pairs $(\Gamma,\gamma)$ where $\Gamma\in\mathcal B$ and $\gamma$ is a sub-object of $\Gamma$, in a sense to be precised. The original motivation came from renormalisation in Quantum Field Theory, more precisely from the formulation of the Bogoliubov-Parasiuk-Hepp-Zimmermann algorithm one can read in Physics textbooks prior to the Connes-Kreimer Hopf-algebraic formalism, such as \cite[Chap. 8]{S1991}. The Feynman rules consist, for each one-particle-irreducible Feynman graph $\Gamma$ of the theory at stake, in integrating the corresponding function $G_\Gamma$ with respect to internal momenta, thus producing a function $F_\Gamma=I_\Gamma(G_\Gamma)$ depending on the external momenta only. These integrals are notoriously divergent in general, and therefore ask for regularisation and renormalisation.  The renormalisation procedure (see e.g. \cite[\S\ 8.1]{S1991}) makes use of the operators $I_{\Gamma,\gamma}$ consisting in integrating (in a suitable regularised sense) only with respect to the internal momenta of a locally 1PI subgraph $\gamma$ of $\Gamma$. To be precise, $I_{\Gamma,\gamma}$ is obtained from the subtraction operator $M_\gamma^{a_\gamma}$ defined by Equation (8.3a) in \cite{S1991} by replacing the Taylor truncation $\mathcal T_{q^\gamma}^{a_\gamma}$ by the identity operator.\\
	
	For any locally 1PI subgraph $\delta$ of $\gamma$ it is possible to proceed in two steps, which gives
	$$I_{\Gamma,\gamma}=I_{\Gamma/\delta,\,\gamma/\delta}\circ I_{\Gamma,\delta},$$
with the usual notation to denote contracted graphs. The operator $I_{\Gamma,\gamma}$ is the identity if the subgraph $\gamma$ is loopless, and $I_{\Gamma,\Gamma}$ is the full integration $I_\Gamma$. The identity above in the case $\gamma=\Gamma$ can be rewritten as
	$$I_\Gamma=I_{\Gamma/\delta}\circ I_{\Gamma,\delta}.$$

The pairs $(\Gamma,\gamma)$ as above generate the doubling bialgebra of the Connes-Kreimer Hopf algebra, the degree being the loop number of the subgraph $\gamma$. The second projection $(\Gamma,\gamma)\mapsto\gamma$ is a graded bialgebra morphism \cite{acg3b} onto the Connes-Kreimer Hopf algebra of specified graphs \cite{acg3a}.\\
	
	The doubling procedure has been adapted by M.~Belhaj Mohamed and the second author to the Hopf algebras of rooted forests \cite{acg4}, and shows in the present work to be relevant also in the finite topology framework.\\
	 
	Recall (see e.g. \cite{acg15,acg16,acg10}) that a topology on a finite set $X$ is given by the
family $\mathcal{T}$ of open subsets of $X$, subject to the three following axioms:
	\begin{itemize}
		\item $\hbox{\o}\in\mathcal{T}$, $X\in\mathcal{T}$,
		\item The union of (a finite number of) open subsets is an open subset,
		\item The intersection of a finite number of open subsets is an open subset.
	\end{itemize}
	Any topology $\mathcal{T}$ on $X$ defines a quasi-order (i.e. a reflexive transitive relation) denoted by $\leq_{\mathcal{T}}$ on $X$:
	\begin{equation}
	x\leq_{\mathcal{T}}y\Longleftrightarrow \hbox{ any open subset containing $x$ also contains $y$}.
	\end{equation}
	Conversely, any quasi-order $\leq$ on $X$ defines a topology $\mathcal{T}_{\leq}$ given by its upper ideals, i.e. subsets $Y\subset X$ such that ($y\in Y$ and $y\leq z$) $\implies z\in Y$. Both operations are inverse to each other:\\
	\begin{equation}
	\leq_{\mathcal{T}_{\leq}}= \leq,\hspace*{2cm} \mathcal{T}_{\leq_{\mathcal{T}}}=\mathcal{T}.
	\end{equation}
	Hence there is a natural bijection between topologies and quasi-orders on
	a finite set $X$.
	Any quasi-order (hence any topology $\mathcal{T}$ ) on $X$ gives rise to an equivalence relation:
	\begin{equation}
	x \sim_{\mathcal{T}}y\Longleftrightarrow \left( x\leq_{\mathcal{T}}y \hbox{ and } y\leq_{\mathcal{T}}x \right) .
	\end{equation}
More on finite topological spaces can be found in \cite{acg3,acg9,acg14, acg16}.\\

	Let us recall the construction from \cite{acg9} of two bimonoids \cite{acg1,acg2} in cointeraction on the linear species of finite topological spaces, which orginated from a previous Hopf-algebraic approach \cite{acg11,acg12}. Let $\mathcal{T}$ and $\mathcal{T}'$ be two topologies on a finite set $X$. We say that $\mathcal{T}'$ is finer than $\mathcal{T}$, and we write $\mathcal{T}'\prec \mathcal{T}$, when any open subset for $\mathcal{T}$ is an open subset for $\mathcal{T}'$. This is equivalent to the fact that for any $x,y\in X$, $x\le_{\mathcal{T}'}y\Rightarrow x\le_{\mathcal{T}}y$.\\
	
	The \textsl{quotient} $\mathcal{T}/\mathcal{T}'$ of two topologies $\mathcal{T}$ and $\mathcal{T}'$ with $\mathcal{T}'\prec \mathcal{T}$ is defined as follows (\cite[Paragraph 2.2]{acg10}): The associated quasi-order $\le_{\mathcal{T}/\mathcal{T}'}$ is the transitive closure of the relation $\mathcal{R}$ defined by:
	\begin{equation}
	x\mathcal{R} y\Longleftrightarrow (x\leq_{\mathcal{T}} y\hbox{ or }y\leq_{\mathcal{T}'} x).
	\end{equation}
	\\
	Recall that a linear species is a contravariant functor from the category
	of finite sets with bijections into the category of vector spaces (on some
	field $\mathbf{k}$). The tensor product of two species  $\mathbb{E}$ and $\mathbb{F}$ is given by
	\begin{equation}
	(\mathbb{E}\otimes \mathbb{F})_X=\bigoplus_{Y\sqcup Z= X}\mathbb{E}_{Y}\otimes\mathbb{F}_{Z},
	\end{equation}
	where the notation $\sqcup $ stands for disjoint union. The species $\mathbb{T}$ of finite topological spaces is defined as follows: For any finite set $X$, $\mathbb{T}_X$  is
	the vector space freely generated by the topologies on $X$. For any bijection
	$\varphi : X \longrightarrow  X^{\prime } $, the isomorphism $\mathbb{T}_{\varphi } : \mathbb{T}_{ X^{\prime}} \longrightarrow \mathbb{T}_X$ is defined by the obvious
	relabelling:
	$$\mathbb{T}_{\varphi }(\mathcal{T})=\{ \varphi^{-1}(Y), Y\in  \mathcal{T}  \}$$
	for any topology $\mathcal{T}$ on $X^{\prime }$.\\
	
	\noindent For any finite set $X$, let us recall from \cite{acg10} the coproduct $\Gamma$ on $\mathbb{T}_X$:
	\begin{equation}
	 \Gamma(\mathcal{T})=\sum \limits_{\underset{}{\mathcal{T}^{\prime}\soprec \mathcal{T}}}\mathcal{T}^{\prime}\otimes \mathcal{T}/ \mathcal{T}^{\prime}.
	 \end{equation}
	The sum runs over topologies $\mathcal{T}^{\prime}$ which are $\mathcal{T}$-admissible, i.e 
\begin{itemize}
	\item finer than $\mathcal{T}$,
	\item such that $\mathcal{T}^{\prime}_{|Y}=\mathcal{T}_{|Y}$ for any subset $Y\subset X$ connected for the topology $\mathcal{T}^{\prime}$,\
	\item such that for any $x, y \in X$,
	\begin{equation}
	 x \sim_{\mathcal{T}/ \mathcal{T}^{\prime}} y \iff x \sim_{\mathcal{T}^{\prime}/ \mathcal{T}^{\prime}} y.
	 \end{equation}
\end{itemize}
	A commutative monoid structure (\cite[Paragraph 2.3]{acg10}) on
	the species of finite topologies is defined as follows: for any pair $X_1, X_2$ of
	finite sets we introduce
	\begin{align*}
	m:\mathbb{T}_{X_1}\otimes \mathbb{T}_{X_2} \longrightarrow \mathbb{T}_{X_1\sqcup X_2} \\
	\mathcal{T}_1\otimes \mathcal{T}_2\longmapsto \mathcal{T}_1\mathcal{T}_2,
	\end{align*}
	where $\mathcal{T}_1\mathcal{T}_2$ is the disjoint union topology characterised by $Y\in \mathcal{T}_1\mathcal{T}_2$ if and only if $Y\cap X_1\in \mathcal{T}_1$ and
	$Y\cap X_2\in \mathcal{T}_2$. The unit is given by the unique topology on the empty set.\\
	
	For any topology $\mathcal{T}$ on a finite set $X$ and for any subset $Y \subset X$, we
	denote by $\mathcal{T}_{|Y}$ the restriction of $\mathcal{T}$ to $Y$. It is defined by:
	$$\mathcal{T}_{|Y}= \left\lbrace Z\cap Y, Z\in \mathcal{T} \right\rbrace. $$
	The external coproduct $\Delta$ on $\mathbb{T}$ is defined as follows:
	\begin{eqnarray*}
		\Delta:\mathbb{T}_X&\longrightarrow& (\mathbb{T}\otimes \mathbb{T})_X=\bigoplus_{Y\sqcup Z= X}\mathbb{T}_{Y}\otimes\mathbb{T}_{Z}\\
		\mathcal{T}&\longmapsto& \sum_{Y\in  \mathcal{T}}\mathcal{T}_{|X\backslash Y}\otimes \mathcal{T}_{|Y}.
	\end{eqnarray*}
	The internal and external coproducts are compatible, i.e. the following diagram commutes for any finite $X$.
	
	\diagramme{
		\xymatrix{
			\mathbb{T}_X\ar[rr]^{\Gamma}\ar[d]_{\Delta} &&
			\mathbb{T}_X\otimes\mathbb{T}_X\ar[d]^{I\otimes\Delta}\\
			(\mathbb{T}\otimes\mathbb{T})_X\ar[dr]_{\Gamma\otimes\Gamma}&&
			\mathbb{T}_X\otimes(\mathbb{T}\otimes\mathbb{T})_X\\
			&\bigoplus\limits_{Y\subset X}\mathbb{T}_{X\backslash Y}\otimes\mathbb{T}_{X\backslash Y}\otimes\mathbb{T}_Y\otimes \mathbb{T}_Y\ar[ur]_{m^{1,3}}&
		}
	}
	
	\noindent Now consider the graded vector space:\\
	\begin{equation}
	\mathcal{H}=\overline{\mathcal{K}}(\mathbb{T})=\bigoplus \limits_{\underset{}{n\geq 0}} \mathcal{H}_n
	\end{equation}
	where $\mathcal{H}_0 =\mathbf{k}.1$, and where $\mathcal{H}_n$ is the linear span of topologies on $\left\{ 1, . . . , n \right\}$
	when $n \geq 1$, modulo the action of the symmetric group $S_n$. The vector space $\mathcal{H}$ can be seen as the quotient of the species $\mathbb{T}$ by the "forget the labels" equivalence relation: $\mathcal{T} \sim \mathcal{T}^{\prime}$  if $\mathcal{T}$ $\left(\text{resp.} \mathcal{T}^{\prime} \right)$ is a topology on a finite
	set $X$ (resp. $X'$), such that there is a bijection from $X$ onto $X^{\prime}$ which is a homeomorphism with respect to both topologies. The functor $\overline{\mathcal{K}}$ from linear species to graded vector spaces thus obtained is intensively studied in (\cite[chapter 15]{acg1}) under the name "bosonic Fock functor". This  naturally leads to the following:\\
	\begin{itemize}
	    \item $(\mathcal{H},m,\Delta )$ is a commutative connected Hopf algebra, graded by the number of elements.
	    \item $(\mathcal{H}, m, \Gamma  )$ is a commutative bialgebra, graded by the number of equivalence classes minus the number of connected components.
	    \item $(\mathcal{H},m,\Delta )$ is a comodule-bialgebra on $(\mathcal{H},m,\Gamma )$. In particular the following diagram of unital algebra morphisms commutes:
	\end{itemize}
	\diagramme{
		\xymatrix{
			\mathcal{H}\ar[rr]^{\Gamma}\ar[d]_{\Delta} &&
			\mathcal{H}\otimes\mathcal{H}\ar[d]^{I\otimes\Delta}\\
			\mathcal{H}\otimes\mathcal{H}\ar[dr]_{\Gamma\otimes\Gamma}&&
			\mathcal{H}\otimes\mathcal{H}\otimes\mathcal{H}\\
			&\mathcal{H}\otimes\mathcal{H}\otimes\mathcal{H}\otimes \mathcal{H}\ar[ur]_{m^{1,3}}&
		}
	}
	
	On the vector space freely generated by rooted forests, Connes and Kreimer define in \cite{acg4, acg5, acg7} a graded bialgebra structure defined using allowable cuts. In \cite{acg8}, D.~Calaque, K.~Ebrahimi-Fard and the second author introduced bases of a graded Hopf algebra structure defined using contractions of trees.
M. Belhaj Mohamed and the second author introduced in \cite{acg6} the doubling of these two spaces and they built two bialgebra structures on these spaces, which are in interaction. They have also shown that two bialgebra satisfied a commutative diagram similar to the diagram of \cite{acg8} in the case of rooted trees Hopf algebra, and in the case of directed graphs without cycles \cite{acg13}.
\\
\\
	
	In Section \ref{sect:doubling} of this paper, we define two different doubling species $\mathbb{D}$ and $\widetilde{\mathbb{D}}$ of the species $\mathbb{T}$. For later use will also consider  $\mathcal{D}=\overline{\mathcal{K}}(\mathbb{D})$ and $\widetilde{\mathcal{D}}=\overline{\mathcal{K}}(\widetilde{\mathbb{D}})$. The species $\mathbb{D}$ is defined as follows: For any finite set $X$, $\mathbb{D}_X$ is the vector space spanned by the pairs $(\mathcal{T},Y)$  where $\mathcal{T}$ is a
	topology on $X$ and $Y \in \mathcal{T}$. Similarly, $\widetilde{\mathbb{D}}_X$ is the vector space spanned by the ordered pairs  $(\mathcal{T}, \mathcal{T}^{\prime} )$
	where $\mathcal{T}$ is a topology on $X$, and $\mathcal{T}^{\prime}\oprec \mathcal{T}$.
	We prove that there exist graded bimonoid  structures on $\mathbb{D}_X$ and $\widetilde{\mathbb{D}}_X$, where the external and internal coproducts are defined respectively by
	\begin{equation}
	\Delta(\mathcal{T}, Y)=\sum \limits_{\underset{}{Z\in \mathcal{T}_{|Y}}}(\mathcal{T}_{|Z}, Z)\otimes (\mathcal{T}_{|X\backslash Z}, Y\backslash Z),
	\end{equation}
	for all $(\mathcal{T}, Y) \in \mathbb{D}_X$, and
	\begin{equation}
	\Gamma (\mathcal{T},\mathcal{T}^{\prime})=\sum \limits_{\underset{}{\mathcal{T}^{\prime\prime}\soprec  \mathcal{T}^{\prime}}}(\mathcal{T},\mathcal{T}^{\prime\prime})\otimes (\mathcal{T}/ \mathcal{T}^{\prime\prime},\mathcal{T}^{\prime}/ \mathcal{T}^{\prime\prime}).
	\end{equation}
	 for all $(\mathcal{T}, \mathcal{T}^{\prime}) \in \mathbb{\widetilde{D}}_X$. We show the inclusions $\Delta(\mathbb{D}_X ) \subset (\mathbb{D}\otimes \mathbb{D})_X=\bigoplus \limits_{Y\sqcup Z= X}\mathbb{D}_{Y}\otimes\mathbb{D}_{Z}$ and $\Gamma(\widetilde{\mathbb{D}}_X ) \subset \widetilde{\mathbb{D}}_X \otimes \widetilde{\mathbb{D}}_X$, and that $\Delta$ and $\Gamma$ are coassociative. It turns out that only the internal coproduct $\Gamma$ is counital.\\

	In Section \ref{sect:comodule}, after a reminder of the main results of \cite{acg10}, we show an important restriction result, namely the notion of $\mathcal T$-admissibility is stable under restriction to any subset (Proposition \ref{restriction}), and we prove that $\mathbb{D}_X$ admits a comodule structure on $\widetilde{\mathbb{D}}_X$ given by the coaction $\Phi : \mathbb{D}_X \longrightarrow \widetilde{\mathbb{D}}_X \otimes \mathbb{D}_X$, which is defined for all $(\mathcal{T}, Y) \in \mathbb{D}_X$ by: 
	$$\Phi (\mathcal{T},Y)=\sum \limits_{\underset{\scriptstyle\mathcal{T}^{\prime}_{|X\backslash Y}=D_{X\backslash Y,\mathcal{T}'}}{\mathcal{T}^{\prime}\soprec  \mathcal{T},\,Y\in \mathcal{T}/ \mathcal{T}' }}(\mathcal{T},\mathcal{T}^{\prime})\otimes (\mathcal{T}/ \mathcal{T}^{\prime},Y)$$
	where, for any topology $\mathcal{T}$ on a finite set $X$, the finest $\mathcal{T}$-admissible topology is denoted by $D_{X,\mathcal{T}}$. The connected components of $D_{X,\mathcal{T}}$ are the equivalence classes of $\mathcal{T}$, and $D_{X,\mathcal{T}}$ restricted to each connected component is the coarse topology. For any $Y\subset X$, we note $D_{Y,\mathcal{T}}$ for $D_{Y,\mathcal{T}_{|Y}}$.
	\begin{remark}\rm
		We obviously have $d(D_{X,\mathcal{T}})=0$ where $d$ is the grading given by the number of equivalence classes minus the number of connected components \cite{acg10}. We also clearly have
		$$\mathcal T/D_{X,\mathcal{T}}=\mathcal T.$$
	\end{remark}
	In Section \ref{sect:assoc}, we construct an associative product on $\mathbb{D}$ given by $\ast : \mathbb{D} \otimes \mathbb{D} \longrightarrow \mathbb{D}$, defined for all $(\mathcal{T}_1,Y_1) \in \mathbb{D}_{X_1}$ and $(\mathcal{T}_2,Y_2) \in \mathbb{D}_{X_2}$, (where $X_1$ and $X_2$ are two finite sets) by:
	\[(\mathcal{T}_1,Y_1)\ast (\mathcal{T}_2,Y_2) \longmapsto \begin{cases}(\mathcal{T}_1,Y_1 \sqcup Y_2) & \text{if $X_2=X_1\setminus Y_1$ and $\mathcal{T}_2={\mathcal{T}_1}\restr {X_2}$}\\
	0 & \text{if not}. \end{cases}\]
	This product is obtained by dualising the restriction of the coproduct $\Delta $ to $\mathbb{D}_X$, identifying $\mathbb{D}_X$ with its graded dual using the basis
	$\left\{ (\mathcal{T},Y), \mathcal{T} \text{ topology and } Y\in \mathcal{T} \right\}$. We accordingly construct a second associative algebra structure on $\widetilde{\mathbb{D}}_X$ by dualising the restriction of the coproduct $\Gamma$ to $\widetilde{\mathbb{D}}_X$, yielding the associative product $\divideontimes : \widetilde{\mathbb{D}}_X \otimes \widetilde{\mathbb{D}}_X \longrightarrow \widetilde{\mathbb{D}}_X$, defined by:
	\[(\mathcal{T}_1,\mathcal{T}^{\prime}_1)\divideontimes (\mathcal{T}_2,\mathcal{T}^{\prime}_2) \longmapsto \begin{cases}(\mathcal{T}_1,\mathcal{U}) & \text{if $\mathcal{T}_2=\mathcal{T}_1/ \mathcal{T}^{\prime}_1$}\\
	0 & \text{if not}, \end{cases}\]
	where $\mathcal{U}$ is defined by $\mathcal{T}^{\prime}_2=\mathcal{U}/\mathcal{T}^{\prime}_1$.\\
	\\
	Finally, we define in Section \ref{sect:comp} a new map
	$$\xi : \widetilde{\mathbb{D}}_X \otimes \bigoplus \limits_{Y\sqcup Z= X}\mathbb{D}_{Y}\otimes\mathbb{D}_{Z} \longrightarrow  \widetilde{\mathbb{D}}_X \otimes \bigoplus \limits_{Y\sqcup Z= X}\mathbb{D}_{Y}\otimes\mathbb{D}_{Z}$$
by:
$$\xi\big( (\mathcal{T},\mathcal{T}^{\prime})\otimes (\mathcal{T}_1,Y_1) \otimes (\mathcal{T}_2, Y_2)\big)= (\mathcal{T}\restr{Y_1}\mathcal{T}\restr{X\backslash Y_1},\mathcal{T}^{\prime}\restr{Y_1}\mathcal{T}^{\prime}\restr{X\backslash Y_1})\otimes (\mathcal{T}_1,Y_1) \otimes (\mathcal{T}_2, Y_2).$$

	We prove that the coaction $\Phi $ and the map $\xi $ make the following diagram commute:
	$$
\xymatrix{
\mathbb{D}_X \ar[rr]^\Phi \ar[d]_{\Delta} && \widetilde{\mathbb{D}}_X \otimes \mathbb{D}_X \ar[d]^{Id \otimes \Delta}\\
	(\mathbb{D}\otimes\mathbb{D})_X \ar[d]_{\Phi \otimes \Phi } && \widetilde{\mathbb{D}}_X \otimes (\mathbb{D} \otimes \mathbb{D})_X \ar[d]^{\xi}\\
\bigoplus \limits_{Y\subset X}\widetilde{\mathbb{D}}_Y \otimes \mathbb{D}_Y \otimes \widetilde{\mathbb{D}}_{X\backslash Y} \otimes \mathbb{D}_{X\backslash Y} \ar[rr]_{m^{1,3} } && \widetilde{\mathbb{D}}_X \otimes (\mathbb{D} \otimes \mathbb{D})_X
}
$$
	Applying the functor $\overline{\mathcal{K}}$ leads to the diagram:
	$$
\xymatrix{
\mathcal{D} \ar[rr]^\Phi \ar[d]_{\Delta} && \widetilde{\mathcal{D}} \otimes \mathcal{D} \ar[d]^{Id \otimes \Delta}\\
	\mathcal{D}\otimes\mathcal{D} \ar[d]_{\Phi \otimes \Phi } && \widetilde{\mathcal{D}} \otimes \mathcal{D} \otimes \mathcal{D} \ar[d]^{\xi}\\
\widetilde{\mathcal{D}} \otimes \mathcal{D} \otimes \widetilde{\mathcal{D}} \otimes \mathcal{D} \ar[rr]_{ m^{1,3} } &&\widetilde{\mathcal{D}} \otimes \mathcal{D} \otimes \mathcal{D}
}
$$
	where we have written $\Delta$ for $\overline{\mathcal{K}}(\Delta)$ and so on. All arrows of this diagram are algebra morphisms.\\
	
	As any oriented graph gives rise to a quasi-poset structure, hence a topology, on its set of vertices,  we plan in a future work to adapt the results obtained here to Hopf algebras and bialgebras of oriented graphs \cite{acg13}.
	
	\section{Doubling bialgebras of finite topologies}\label{sect:doubling}
	Let $X$ be any finite set, and $\mathbb{D}_X$ be the vector space spanned by the pairs $(\mathcal{T}, Y )$
	where $\mathcal{T}$ is a topology, and $Y \in \mathcal{T}$. We define the coproduct $\Delta$ by:
	\begin{align*}
	\Delta:\mathbb{D}_X&\longrightarrow (\mathbb{D} \otimes \mathbb{D})_X=\bigoplus \limits_{Z\subset X}\mathbb{D}_{Z}\otimes\mathbb{D}_{X\backslash Z}\\  	
	(\mathcal{T},Y)&\longmapsto \sum_{Z\in  \mathcal{T}\srestr{Y}}(\mathcal{T}\restr{Z}, Z)\otimes (\mathcal{T}\restr{X\backslash Z},Y\backslash Z).
	\end{align*}
	\begin{theorem}
		$\mathbb{D}$ is a commutative graded connected bimonoid, and $\mathcal{D}=\overline{\mathcal{K}}(\mathbb{D})$ is a commutative graded bialgebra.
	\end{theorem}
	\begin{proof}
		To show that $\mathbb{D}$ is a bimonoid \cite{acg1}, it is necessary to show that $\Delta$ is coassociative, and that the species coproduct $\Delta$ and the product defined by:
		$$(\mathcal{T}_1,Y_1)(\mathcal{T}_2,Y_2)=(\mathcal{T}_1\mathcal{T}_2,Y_1\sqcup Y_2)$$ 
		are compatible.
		The unit $\mathbf 1$ is identified to the empty topology, \ignore{the counit $\epsilon$ is given by\\ $\epsilon (\mathcal{T},Y) = 1$ if $Y=\emptyset $ and $0$ otherwise, }and the grading is given by:
		\begin{equation}
		d(\mathcal{T},Y)=|Y|.
		\end{equation}
		The associativity of the product is given by the direct computation:
$$(\mathcal{T}_1\mathcal{T}_2, Y_1\sqcup Y_2)(\mathcal{T}_3,Y_3) =(\mathcal{T}_1\mathcal{T}_2\mathcal{T}_3,Y_1\sqcup Y_2\sqcup Y_3)= (\mathcal{T}_1,Y_1)(\mathcal{T}_2\mathcal{T}_3, Y_2\sqcup Y_3).$$
		The coassociativity of coproduct $\Delta$ is also straightforwardly checked:
		\begin{align*}
		(\Delta \otimes id)\Delta (\mathcal{T},Y)&=(\Delta \otimes id)\left( \sum \limits_{\underset{}{Z\in \mathcal{T}\srestr{Y}}}(\mathcal{T}\restr{Z}, Z)\otimes (\mathcal{T}\restr{X\backslash Z},Y\backslash Z)\right) \\
		&=\sum_{W\in \mathcal{T}\srestr{Z},\  Z\in \mathcal{T}\srestr{Y}}(\mathcal{T}\restr{W},W)\otimes (\mathcal{T}\restr{Z\backslash W}, Z \backslash W)\otimes (\mathcal{T}\restr{X\backslash Z},Y\backslash Z),
		\end{align*}
		and
		\begin{align*}
		(id \otimes \Delta )\Delta (\mathcal{T},Y)&=(id \otimes \Delta)\left( \sum_{Z\in \mathcal{T}\srestr{Y}}(\mathcal{T}\restr{Z}, Z)\otimes (\mathcal{T}\restr{X\backslash Z},Y\backslash Z) \right) \\
		&=\sum_{U\in \mathcal{T}\srestr{Y\backslash Z},\ Z\in \mathcal{T}\srestr{Y}}(\mathcal{T}\restr{Z},Z)\otimes (\mathcal{T}\restr{U},U)\otimes \big(\mathcal{T}\restr{X\backslash (Z \sqcup U)},Y\backslash (Z \sqcup U)\big).
		\end{align*}
		Coassociativity then comes from the obvious fact that $(W, Z )\longmapsto (W, Z\backslash W)$ is a bijection from the set of pairs $(W,Z)$ with $Z\in \mathcal{T}\restr{Y}$ and $W\in \mathcal{T}\restr{Z}$, onto the set of pairs $(W, U)$ with $W\in \mathcal{T}\restr{Y}$ and $U\in \mathcal{T}\restr{Y\backslash W}$. The inverse map is given by $(W, U)\longmapsto (W, W\sqcup U)$. Finally, we show immediately that
		$$\Delta \big( (\mathcal{T}_1, Y_1)(\mathcal{T}_2, Y_2)\big) =\Delta(\mathcal{T}_1\mathcal{T}_2,Y_1\sqcup Y_2)=\Delta(\mathcal{T}_1, Y_1)\Delta (\mathcal{T}_2, Y_2).$$
		\end{proof}
\begin{remark}\rm
The bimonoid $\mathbb D$ is not counitary, because $(\mathcal T,Y)\otimes \mathbf 1$ never occurs in $\Delta(\mathcal T,Y)$ unless $Y=X$.
\end{remark}
	
	Let $\widetilde{\mathbb{D}}_X$ be the vector space spanned by the ordered pairs $(\mathcal{T},\mathcal{T}^{\prime})$  where $\mathcal{T}$ is a
	topology on $X$ and $\mathcal{T}^{\prime}\oprec \mathcal{T}$. We define the coproduct $\Gamma$ for all $(\mathcal{T}, \mathcal{T}^{\prime}) \in \widetilde{\mathbb{D}}_X$ by:
	$$\Gamma(\mathcal{T}, \mathcal{T}^{\prime})=\sum_{\mathcal{T}^{\prime \prime}\soprec \mathcal{T}^{\prime}}(\mathcal{T}, \mathcal{T}^{\prime \prime})\otimes (\mathcal{T}/ \mathcal{T}^{\prime \prime}, \mathcal{T}^{\prime}/ \mathcal{T}^{\prime \prime}).$$
	\begin{lemma}\label{lemme-sandwich}(\cite[Propostion 2.7]{acg10})
		Let $\mathcal{T}$ and $\mathcal{T}^{\prime \prime}$ be two topologies on $X$. If $ \mathcal{T}^{\prime \prime}\oprec \mathcal{T}$,
		then $\mathcal{T}^{\prime}\longmapsto \mathcal{T}^{\prime}/\mathcal{T}^{\prime \prime}$ is a bijection from the set of topologies $\mathcal{T}^{\prime}$ on $X$ such that $\mathcal{T}^{\prime \prime}\oprec \mathcal{T}^{\prime}\oprec \mathcal{T}$ , onto the set of topologies $\mathcal{U}$ on $X$ such that $\mathcal{U}\oprec \mathcal{T}/\mathcal{T}^{\prime \prime}$. 
	\end{lemma}
	\begin{theorem}
		$\widetilde{\mathbb{D}}$ is a commutative graded bimonoid, and $\widetilde{\mathcal{D}}=\overline{\mathcal{K}}(\widetilde{\mathbb{D}})$ is a graded bialgebra.
	\end{theorem}
	\begin{proof}
		To show that $\widetilde{\mathbb{D}}$ is a bimonoid, it is necessary to show that $\Gamma$ is coassociative and that the species coproduct $\Gamma$ and the product defined by:
		$$m\big( (\mathcal{T}_1,\mathcal{T}^{\prime}_1)(\mathcal{T}_2,\mathcal{T}^{\prime}_2)\big) =(\mathcal{T}_1\mathcal{T}_2,\mathcal{T}^{\prime}_1\mathcal{T}^{\prime}_2)$$ 
		are compatible.	
		The unit $\mathbf 1$ is identified to the empty topology, the counit $\epsilon$ is given by $\epsilon (\mathcal{T},\mathcal{T}^{\prime}) = \epsilon (\mathcal{T}^{\prime})$ and the grading is given by:
		\begin{equation}
		d(\mathcal{T},\mathcal{T}^{\prime})=d(\mathcal{T}^{\prime}),
		\end{equation}
		where the grading d on the right-hand side has been defined in the introdution.\\
		We now calculate:\\
		\begin{align*}
		(\Gamma \otimes id)\Gamma (\mathcal{T}, \mathcal{T}^{\prime})&=(\Gamma \otimes id)\left( \sum \limits_{\underset{}{\mathcal{T}^{\prime \prime}\soprec \mathcal{T}^{\prime}}}(\mathcal{T}, \mathcal{T}^{\prime \prime})\otimes (\mathcal{T}/ \mathcal{T}^{\prime \prime}, \mathcal{T}^{\prime}/ \mathcal{T}^{\prime \prime})\right) \\
		&=\sum \limits_{\underset{}{\mathcal{T}^{\prime \prime \prime}\soprec \mathcal{T}^{\prime \prime}\soprec \mathcal{T}^{\prime}}}(\mathcal{T}, \mathcal{T}^{\prime \prime \prime})\otimes (\mathcal{T}/ \mathcal{T}^{\prime \prime \prime}, \mathcal{T}^{\prime \prime}/ \mathcal{T}^{\prime \prime \prime})\otimes (\mathcal{T}/ \mathcal{T}^{\prime \prime}, \mathcal{T}^{\prime}/ \mathcal{T}^{\prime \prime}).
		\end{align*}
		On the other hand;
		\begin{align*}
		(id\otimes \Gamma)\Gamma (\mathcal{T}, \mathcal{T}^{\prime})&=(id\otimes \Gamma)\left( \sum \limits_{\underset{}{\mathcal{T}^{\prime \prime}\soprec \mathcal{T}^{\prime}}}(\mathcal{T}, \mathcal{T}^{\prime \prime})\otimes (\mathcal{T}/ \mathcal{T}^{\prime \prime}, \mathcal{T}^{\prime}/ \mathcal{T}^{\prime \prime})\right) \\
		&=\sum \limits_{\underset{}{\mathcal{T}^{\prime \prime}\soprec \mathcal{T}^{\prime},\hskip1mm \mathcal{T}_1 \soprec \mathcal{T}^{\prime}/ \mathcal{T}^{\prime \prime} }}(\mathcal{T}, \mathcal{T}^{\prime \prime})\otimes (\mathcal{T}/ \mathcal{T}^{\prime \prime}, \mathcal{T}_1)\otimes \big( (\mathcal{T}/ \mathcal{T}^{\prime \prime})/\mathcal{T}_1, (\mathcal{T}^{\prime}/ \mathcal{T}^{\prime \prime})/\mathcal{T}_1\big) \\
		&=\sum \limits_{\underset{}{\mathcal{T}^{\prime \prime}\soprec \mathcal{U}\soprec \mathcal{T}^{\prime}}}(\mathcal{T}, \mathcal{T}^{\prime \prime})\otimes (\mathcal{T}/ \mathcal{T}^{\prime \prime}, \mathcal{U}/ \mathcal{T}^{\prime \prime})\otimes (\mathcal{T}/ \mathcal{U}, \mathcal{T}^{\prime}/ \mathcal{U}).
		\end{align*}
		The result then comes from Lemma 2.1.
		Hence, $(\Gamma \otimes id)\Gamma =(id\otimes \Gamma )\Gamma $, and consequently $\Gamma $ is coassociative. Finally we have directy:\\
		$$\Gamma \big( (\mathcal{T}_1, \mathcal{T}^{\prime}_1)(\mathcal{T}_2, \mathcal{T}^{\prime}_2)\big) =\Gamma(\mathcal{T}_1, \mathcal{T}^{\prime}_1)\Gamma(\mathcal{T}_2, \mathcal{T}^{\prime}_2).$$
	\end{proof}
	\begin{proposition}
		The second projection
		\begin{align*}
		P_2&:\widetilde{\mathbb{D}}\longrightarrow \mathbb{T}\\
		&(\mathcal{T}, \mathcal{T}^{\prime})\longmapsto \mathcal{T}^{\prime}
		\end{align*}
		is a bimonoid morphism with respect to the internal coproducts.
	\end{proposition}
	\begin{proof}
		The fact that $P_2$ respects the product is trivial. It suffices to show that $P_2$ is a coalgebra morphism for any finite set $X$, analogously to Proposition 1, i.e, $P_2$ verifies the following commutative diagram:
		$$
		\xymatrix{
			\widetilde{\mathbb{D}}_X \ar[rr]^{P_2} \ar[d]_{\Gamma} && \mathbb{T}_X\ \ar[d]^{\Gamma}\\
			\widetilde{\mathbb{D}}_X\otimes \widetilde{\mathbb{D}}_X \ar[rr]_{P_2\otimes P_2} && \mathbb{T}_X\otimes \mathbb{T}_X
		}
		$$
		which can be seen by direct computation:
		\begin{align*}
		\Gamma \circ P_2(\mathcal{T}, \mathcal{T}^{\prime})&=\Gamma (\mathcal{T}^{\prime})\\
		&=\sum \limits_{\underset{}{\mathcal{T}^{\prime \prime}\soprec \mathcal{T}^{\prime}}}\mathcal{T}^{\prime \prime}\otimes \mathcal{T}^{\prime}/ \mathcal{T}^{\prime \prime}\\
		&=\sum \limits_{\underset{}{\mathcal{T}^{\prime \prime}\soprec \mathcal{T}^{\prime}}}P_2(\mathcal{T}, \mathcal{T}^{\prime \prime})\otimes P_2(\mathcal{T}/ \mathcal{T}^{\prime \prime}, \mathcal{T}^{\prime}/ \mathcal{T}^{\prime \prime})\\
		&=(P_2\otimes P_2)\Gamma (\mathcal{T}, \mathcal{T}^{\prime}).
		\end{align*}
	\end{proof}
	\section{Comodule-Hopf algebra structure}\label{sect:comodule}
	\subsection{Comodule-Hopf algebra structure on the commutative Hopf algebra of finite topologies  \cite{acg10}}
	F. Fauvet, L. Foissy and the second author have studied the Hopf algebra $(\mathcal{H},m,\Delta )$ as a comodule-Hopf algebra on the bialgebra $(\mathcal{H},m,\Gamma )$, where $\mathcal H=\overline{\mathcal K}(\mathbb T)$. Here the notations $m,\Delta,\Gamma$ are shorthands for $\overline{\mathcal K}(m), \overline{\mathcal K}(\Delta), \overline{\mathcal K}(\Gamma)$ respectively. The coaction is the map $ \overline{\mathcal{K}}(\phi): \mathcal{H} \longrightarrow \mathcal{H}\otimes \mathcal{H}$ where $\phi$ is defined as follows:
	$$\phi (\mathcal{T})=\Gamma (\mathcal{T})=\sum \limits_{\underset{}{\mathcal{T}^{\prime}\soprec \mathcal{T}}}\mathcal{T}^{\prime}\otimes \mathcal{T}/ \mathcal{T}^{\prime}.$$ 
	\begin{proposition}
		 \cite{acg10} The internal and external coproducts are compatible, i.e. the following diagram commutes.
		 
		 	$$
		\xymatrix{
			\mathbb{T} \ar[rr]^\Gamma \ar[d]_{\Delta} && \mathbb{T} \otimes \mathbb{T} \ar[dd]^{Id \otimes \Delta}\\
			(\mathbb{T} \otimes \mathbb{T})_X \ar[d]_{\Gamma \otimes \Gamma } && \\
			\bigoplus \limits_{Y\sqcup Z= X}\mathbb{T}_{Y}\otimes\mathbb{T}_{Y} \otimes \mathbb{T}_{Z}\otimes\mathbb{T}_{Z} \ar[rr]_{m^{1,3}} && \mathbb{T} \otimes (\mathbb{T} \otimes \mathbb{T})_X
		}
		$$
		i.e., the following identity is verified:
		$$(Id\otimes \Delta)\circ \Gamma =m^{1,3}\circ (\Gamma \otimes \Gamma)\circ \Delta ,$$
		where $m^{1,3}:\mathbb{T}_{X_1} \otimes \mathbb{T}_{X_2} \otimes \mathbb{T}_{X_3} \otimes \mathbb{T}_{X_4} \longrightarrow \mathbb{T}_{X_1\sqcup X_3} \otimes \mathbb{T}_{X_2} \otimes \mathbb{T}_{X_4}$ is defined by
		$$m^{1,3}(\mathcal{T}_1\otimes\mathcal{T}_2\otimes \mathcal{T}_3\otimes \mathcal{T}_4)=\mathcal{T}_1\mathcal{T}_3\otimes \mathcal{T}_2\otimes \mathcal{T}_4.$$
\end{proposition}
		Applying the functor $\overline{\mathcal{K}}$ yields the comodule-Hopf algebra structure of $(\mathcal H,m,\Delta)$ on the bialgebra $(\mathcal H,m,\Gamma)$. In particular the diagram above yields the commutative diagram
		$$
		\xymatrix{
			\mathcal{H} \ar[rr]^\Gamma \ar[d]_{\Delta} && \mathcal{H} \otimes \mathcal{H} \ar[dd]^{Id \otimes \Delta}\\
			\mathcal{H} \otimes \mathcal{H} \ar[d]_{\Gamma \otimes \Gamma } && \\
			\mathcal{H} \otimes \mathcal{H} \otimes \mathcal{H} \otimes \mathcal{H} \ar[rr]_{m^{1,3}} && \mathcal{H} \otimes \mathcal{H} \otimes \mathcal{H}
		}
		$$
\subsection{A restriction result}
We prove that the admissibility relation $\oprec$ between two topologies on the same finite set $X$, which is the key to the definition of the inernal coproduct $\Gamma$, is robust under restriction to any subset.
\begin{proposition}\label{restriction}
	Let $\mathcal{T}$ be a topology on a finite set $X$. For any subset $W\subset X$ and for any $\mathcal{T}^{\prime}\oprec \mathcal{T}$ we have $\mathcal{T}^{'}\restr{W}\oprec \mathcal{T}\restr{W}$. 
	\end{proposition}
	\begin{proof}
	Let $\mathcal{T}$ be a topology on a finite set $X$, let $W$ be any subset of $X$, and let $\mathcal{T}^{\prime}\oprec \mathcal{T}$. Let $R$ (resp. $R'$) be the relation defined on $X$ by $aRb \text{ if and only if}\ a\leq_{\mathcal{T}}b\ \text{or }\ b\leq_{\mathcal{T}'} a$ (resp. $aR'b \text{ if and only if}\ a\leq_{\mathcal{T}'}b\ \text{or}\ b\leq_{\mathcal{T}'} a$). We have $\mathcal{T}'\oprec \mathcal{T}$ hence $R'$ implies $R$.
\begin{itemize}
\item The relation $\mathcal{T}^{'}\restr{W}\prec \mathcal{T}\restr{W}$ is obvious.
\item If $Y\subset W$ connected for the topology $\mathcal{T}^{'}$, and $x\in Y$, we have
$$Y = \left\{ y \in W, \text{ there is a chain } xR't_1\cdots R't_nR'y, \text{with}\ t_1,\ldots,t_n \in W
 \right\} .$$
The set $\widetilde{Y} := \left\{ y \in W, \text{ there is a chain } xR't_1\cdots R't_nR'y, \text{with}\ t_1,\ldots, t_n \in X \right\} $ is a connected component of $X$ for the topology $\mathcal{T}'$, so
 $\mathcal{T}^{'}\restr{\widetilde{Y}}=\mathcal{T}\restr{\widetilde{Y}}$, hence a fortiori $\mathcal{T}^{'}\restr{Y}=\mathcal{T}\restr{Y}$, because the inclusion $Y \subset \widetilde{Y}$ holds.\\
 \item Let $x, y \in W$. If $x\sim_{\mathcal{T}^{'}\srestr{W}/{\mathcal{T}^{'}\srestr{W}}} y$  there is $t_1,\ldots,t_n \in W$, $j\in [n]$, $y=t_j$ such that $xR't_1\cdots R't_nR'x$. This implies $xRt_1...Rt_nRx$,
 therefore $x\sim_{\mathcal{T}_{|W}/{\mathcal{T}^{'}_{|W}}} y$. Conversely, if $x\sim_{\mathcal{T}_{|W}/{\mathcal{T}^{'}_{|W}}} y$
  there is $t_1,\ldots,t_n \in W$ and $j\in [n]$ with $y=t_j$, such that $xRt_1\cdots Rt_nRx$. For $A=\left\{ x,t_1,\ldots ,t_n\right\}$, we have for all $a$ and $b$ in $A$, $a\sim_{\mathcal{T}/\mathcal{T}'} b $. Since $\mathcal{T}'\oprec \mathcal{T}$, we have $a\sim_{\mathcal{T}'/\mathcal{T}'} b $, hence $a$ and $b$ in the same connected component $Z$ for the topology $\mathcal{T}'$.\\
  
 We have $\mathcal{T}^{'}\restr{Z}=\mathcal{T}\restr{Z}$ and $A \subset Z$, hence $\mathcal{T}^{'}\restr{A}=\mathcal{T}\restr{A}$. Then for all $a, b \in A$, $aRb$ if and only if $aR'b$, so we have
   $xR't_1\cdots R't_nR'x$, therefore $x\sim_{\mathcal{T}^{'}\srestr{W}/{\mathcal{T}^{'}\srestr{W}}} y$.
\end{itemize}
\end{proof}
\subsection{Comodule structure on the doubling bialgebras of finite topologies}
	For any finite set $X$, we define $\Phi : \mathbb{D}_X \longrightarrow \widetilde{\mathbb{D}}_X \otimes \mathbb{D}_X$, for all 
	$(\mathcal{T}, Y) \in \mathbb{D}_X$ by:
	$$\Phi (\mathcal{T},Y)=\sum_{{\scriptstyle\mathcal{T}^{\prime}\soprec  \mathcal{T},\,Y\in \mathcal{T}/ \mathcal{T}^{\prime},}\atop \scriptstyle\mathcal{T}^{\prime}\srestr{X\backslash Y}=D_{X\backslash Y,\,\mathcal{T}'}}(\mathcal{T},\mathcal{T}^{\prime})\otimes (\mathcal{T}/ \mathcal{T}^{\prime},Y).$$
	The map $\Phi $ is well defined.
	\begin{theorem}
		$\mathbb{D}$ admits a comodule structure on $\widetilde{\mathbb{D}}$ given by $\Phi $.
	\end{theorem}
	\begin{proof}
		The proof amounts to show that the following diagram is commutative for any finite set $X$:
		$$
		\xymatrix{
			\mathbb{D}_X \ar[rr]^{\Phi } \ar[d]_{\Phi } && \widetilde{\mathbb{D}}_X\otimes \mathbb{D}_X \ar[d]^{\Gamma \otimes id }\\
			\widetilde{\mathbb{D}}_X\otimes \mathbb{D}_X \ar[rr]_{id\otimes \Phi } && \widetilde{\mathbb{D}}_X\otimes \widetilde{\mathbb{D}}_X\otimes \mathbb{D}_X
		}
		$$
		Let $(\mathcal{T},Y)\in \mathbb{D}_X$:
		\begin{align*}
		(\Gamma \otimes id)\Phi(\mathcal{T},Y)&=(\Gamma \otimes id)\left( \sum_{{\scriptstyle\mathcal{T}^{\prime}\soprec  \mathcal{T},\,Y\in \mathcal{T}/ \mathcal{T}^{\prime},}\atop \scriptstyle\mathcal{T}^{\prime}\srestr{X\backslash Y}=D_{X\backslash Y,\,\mathcal{T}'}}(\mathcal{T},\mathcal{T}^{\prime})\otimes (\mathcal{T}/ \mathcal{T}^{\prime},Y)\right) \\
		&=\sum_{{\scriptstyle\mathcal{T}^{\prime\prime}\soprec\mathcal{T}^{\prime}\soprec  \mathcal{T},\,Y\in \mathcal{T}/ \mathcal{T}^{\prime},}\atop \scriptstyle\mathcal{T}^{\prime}\srestr{X\backslash Y}=D_{X\backslash Y,\,\mathcal{T}'}}(\mathcal{T}, \mathcal{T}^{\prime \prime })\otimes (\mathcal{T}/ \mathcal{T}^{\prime \prime }, \mathcal{T}^{\prime }/ \mathcal{T}^{\prime  \prime})\otimes (\mathcal{T}/ \mathcal{T}^{\prime }, Y). 
		\end{align*}
		On the other hand, we have
		\begin{align*}
		(id\otimes \Phi )\Phi (\mathcal{T}, Y)&=(id\otimes \Phi )\left( \sum_{{\scriptstyle\mathcal{T}^{\prime}\soprec  \mathcal{T},\,Y\in \mathcal{T}/ \mathcal{T}^{\prime},}\atop \scriptstyle\mathcal{T}^{\prime}\srestr{X\backslash Y}=D_{X\backslash Y,\,\mathcal{T}'}}(\mathcal{T}, \mathcal{T}^{\prime })\otimes (\mathcal{T}/ \mathcal{T}^{\prime }, Y)\right) \\
		&= \sum_{{\scriptstyle\mathcal{T}^{\prime}\soprec  \mathcal{T},\,Y\in \mathcal{T}/ \mathcal{T}^{\prime},}\atop \scriptstyle\mathcal{T}^{\prime}\srestr{X\backslash Y}=D_{X\backslash Y,\,\mathcal{T}'}}
		\sum_{{\scriptstyle\mathcal{T}_1\soprec  \mathcal{T}/\mathcal{T}^\prime,\,Y\in (\mathcal{T}/ \mathcal{T}^{\prime})/\mathcal{T}_1,}\atop \scriptstyle\mathcal{T}_1\srestr{X\backslash Y}=D_{X\backslash Y,\,\mathcal{T}_1}}(\mathcal{T}, \mathcal{T}^{\prime })\otimes (\mathcal{T}/ \mathcal{T}^{\prime}, \mathcal{T}_1)\otimes \left( (\mathcal{T}/ \mathcal{T}^{\prime })/\mathcal{T}_1, Y\right) \\
		&=\sum \limits_{\underset{\scriptstyle\mathcal{U}_{|X\backslash Y}=D_{X\backslash Y,\mathcal{U}} }{ \mathcal{T}^{\prime}\soprec \mathcal{U}\soprec  \mathcal{T},\hskip1mm Y\in \mathcal{T}/ \mathcal{T}^{\prime}}}(\mathcal{T}, \mathcal{T}^{\prime})\otimes (\mathcal{T}/ \mathcal{T}^{\prime }, \mathcal{U}/ \mathcal{T}^{\prime })\otimes (\mathcal{T}/ \mathcal{U},  Y).
		\end{align*}
		Then,
		$$(id\otimes \Phi )\circ \Phi =(\Gamma \otimes id)\circ \Phi,$$
		and consequently $\Phi $ is a coaction.
	\end{proof}
\begin{proposition}
	$\Phi$ is a monoid morphism, i.e. the following diagram is commutative:\\
	$$
			\xymatrix{ 
				\mathbb{D}_X \ar[rr]^{\Phi} && \widetilde{\mathbb{D}}_X \otimes \mathbb{D}_X\\
				(\mathbb{D} \otimes \mathbb{D})_X \ar[d]^{\Phi \otimes \Phi} \ar@{.>}[rr]^{\widetilde{\Phi}} \ar[u]_{m} && \widetilde{\mathbb{D}}_X \otimes (\mathbb{D} \otimes \mathbb{D})_X \ar[u]^{Id \otimes m} \\
				\bigoplus \limits_{Y\subset X}\widetilde{\mathbb{D}}_Y \otimes \mathbb{D}_Y \otimes \widetilde{\mathbb{D}}_{X\backslash Y} \otimes \mathbb{D}_{X\backslash Y} \ar[rr]^{\tau_{23}} && \bigoplus \limits_{Y\subset X}\widetilde{\mathbb{D}}_Y \otimes  \widetilde{\mathbb{D}}_{X\backslash Y} \otimes \mathbb{D}_Y \otimes \mathbb{D}_{X\backslash Y}\ar[u]^{\widetilde{m} \otimes Id \otimes Id} \ar@/_5pc/[uu]_{\widetilde{m} \otimes m}
			}
		$$	
\end{proposition}
\begin{proof}
	
	Let $(\mathcal{T}_1,Y_1)\in \mathbb{D}_{X_1}$ and $(\mathcal{T}_2,Y_2)\in \mathbb{D}_{X_2}$, with $X_1\sqcup X_2=X$. Let $\mathcal{T}^{\prime}_1\oprec \mathcal{T}_1$ and $\mathcal{T}^{\prime}_2\oprec \mathcal{T}_2$. Then $\mathcal{T}^{\prime}_1\mathcal{T}^{\prime}_2\oprec \mathcal{T}_1\mathcal{T}_2$. Conversely, any topology $\mathcal{U}$ on $X$ such that $\mathcal{U}\oprec \mathcal{T}_1\mathcal{T}_2$ can be written $\mathcal{T}^{\prime}_1\mathcal{T}^{\prime}_2$ with $\mathcal{T}^{\prime}_i=\mathcal{U}_{|X_i}$ for $i=1,2$, and we have $\mathcal{T}^{\prime}_i\oprec \mathcal{T}_i$. We have then:
	\begin{align*}
	\Phi \circ m\big( (\mathcal{T}_1,Y_1)\otimes (\mathcal{T}_2, Y_2) \big) =\sum \limits_{\underset{\scriptstyle\mathcal{U}_{|X_1\sqcup X_2\backslash Y_1\sqcup Y_2}=D_{X_1\sqcup X_2\backslash Y_1\sqcup Y_2, \mathcal{U} } }{\mathcal{U}\soprec  \mathcal{T},\hskip1mm Y_1\sqcup Y_2\in \mathcal{T}_1\mathcal{T}_2/ \mathcal{U}}}(\mathcal{T}_1\mathcal{T}_2,\mathcal{U})\otimes \big((\mathcal{T}_1\mathcal{T}_2)/ \mathcal{U},Y_1\sqcup Y_2\big)
	\end{align*}
	On the other hand, we have
		\begin{eqnarray*}
	(\Phi \otimes \Phi )\big( (\mathcal{T}_1,Y_1)\otimes (\mathcal{T}_2, Y_2) \big)=\\ 
	&\hskip -15mm \displaystyle\sum \limits_{\underset{\scriptstyle\mathcal{T}^{\prime}_2\soprec \mathcal{T}_2,\hskip1mm   \mathcal{T}^{\prime}_2\srestr{X_2\backslash Y_2}=D_{X_2\backslash Y_2,\mathcal{T}^{\prime}_2 },\hskip1mm Y_2\in \mathcal{T}_2/ \mathcal{T}^{\prime }_2}{\mathcal{T}^{\prime}_1\soprec \mathcal{T}_1,\hskip1mm   \mathcal{T}^{\prime}_1\srestr{X_1\backslash Y_1}=D_{X_1\backslash Y_1,\mathcal{T}^{\prime}_1},\hskip1mm Y_1\in \mathcal{T}_1/ \mathcal{T}^{\prime}_1 }}(\mathcal{T}_1,\mathcal{T}^{\prime}_1)\otimes (\mathcal{T}_1/ \mathcal{T}^{\prime}_1, Y_1)\otimes (\mathcal{T}_2, \mathcal{T}^{\prime }_2)\otimes(\mathcal{T}_2/\mathcal{T}^{\prime }_2, Y_2),
	\end{eqnarray*}
	therefore 
	\begin{align*}
	(\widetilde{m}\otimes m)\circ \tau_{23}\circ (\Phi \otimes \Phi )&\big((\mathcal{T}_1,Y_1)\otimes (\mathcal{T}_2, Y_2)\big)\\
	&=\sum \limits_{\underset{\scriptstyle\mathcal{T}^{\prime}_2\soprec \mathcal{T}_2,\hskip1mm  \mathcal{T}^{\prime}_2\srestr{X_2\backslash Y_2}=D_{X_2\backslash Y_2,\mathcal{T}^{\prime}_2},\hskip1mm Y_2\in \mathcal{T}_2/ \mathcal{T}^{\prime }_2}{\mathcal{T}^{\prime}_1\soprec \mathcal{T}_1,\hskip1mm   \mathcal{T}^{\prime}_1\srestr{X_1\backslash Y_1}=D_{X_1\backslash Y_1,\mathcal{T}^{\prime}_1},\hskip1mm Y_1\in \mathcal{T}_1/ \mathcal{T}^{\prime}_1 }}(\mathcal{T}_1\mathcal{T}_2,\mathcal{T}^{\prime}_1\mathcal{T}^{\prime }_2)\otimes ((\mathcal{T}_1/ \mathcal{T}^{\prime}_1)(\mathcal{T}_2/\mathcal{T}^{\prime }_2), Y_1\sqcup Y_2)\\
	&=\Phi \circ m\big((\mathcal{T}_1,Y_1)\otimes (\mathcal{T}_2, Y_2)\big).
	\end{align*}
	Hence 	
	$$\Phi \circ m=(\widetilde{m}\otimes m)\circ \tau_{23}\circ (\Phi \otimes \Phi ).$$
\end{proof}
	\section{Associative algebra structures on the doubling spaces}\label{sect:assoc}
	\subsection{Associative product on $\mathcal{D}$}
	For any finite set $X$, recall here that an element $(\mathcal{T}, Y)$ belongs to $\mathbb{D}_X$ if $\mathcal{T}$ is a topology on $X$ and $Y\in \mathcal{T}$.
	\begin{theorem}
		The product $\ast : \mathbb{D}\otimes \mathbb{D}  \longrightarrow \mathbb{D}$ defined for all $(\mathcal{T}_1,Y_1)\in \mathbb{D}_{X_1}$ and $(\mathcal{T}_2,Y_2) \in \mathbb{D}_{X_2}$ by:
		\[(\mathcal{T}_1,Y_1)\ast (\mathcal{T}_2,Y_2) \longmapsto \begin{cases}(\mathcal{T}_1,Y_1 \sqcup Y_2) & \text{if } X_2=X_1\setminus Y_1 \text{ and }  \mathcal{T}_2=\mathcal{T}_1\restr {X_2},\\
		0 & \text {if not} \end{cases}\]
		is associative.
	\end{theorem} 
	\begin{proof}
		Let $(\mathcal{T}_1,Y_1), (\mathcal{T}_2,Y_2)$ and $(\mathcal{T}_3,Y_3)$ be three elements of $\mathbb{D}_{X_1}$, $\mathbb{D}_{X_2}$ and $\mathbb{D}_{X_3}$ respectively. We suppose first that $X_2=X_1\setminus Y_1$ and $\mathcal{T}_2=\mathcal{T}_1\restr {X_2}$, otherwise the result is zero.
		\begin{align*}
		\big( (\mathcal{T}_1,Y_1)\ast (\mathcal{T}_2,Y_2) \big) \ast (\mathcal{T}_3,Y_3)&=(\mathcal{T}_1,Y_1\sqcup Y_2)\ast (\mathcal{T}_3,Y_3)\\
		&=(\mathcal{T}_1, Y_1\sqcup Y_2\sqcup Y_3),
		\end{align*}
		whenever $X_3=X \backslash (Y_1\sqcup Y_2)$ and $\mathcal T_3=\mathcal T_1\restr{X_3}$, the left-hand side vanishing otherwise. Hence,\\
		\[\big( (\mathcal{T}_1,Y_1)\ast (\mathcal{T}_2,Y_2) \big) \ast (\mathcal{T}_3,Y_3)= \begin{cases}(\mathcal{T}_1,Y_1\sqcup Y_2\sqcup Y_3) & \\
		&\hskip-25mm \text{if $X_2=X_1\setminus Y_1$, $X_3=X_1\setminus(Y_1\sqcup Y_2)$, $\mathcal T_2=\mathcal T_1\restr{X_2}$ and $\mathcal T_3=\mathcal T_1\restr{X_3}$}\\
		0 & \text{if not.} \end{cases}\]
On the other hand, we have
		\begin{align*}
		(\mathcal{T}_1,Y_1)\ast \big( (\mathcal{T}_2,Y_2)\ast (\mathcal{T}_3,Y_3) \big) &=(\mathcal{T}_1,Y_1)\ast (\mathcal{T}_2, Y_2 \sqcup Y_3)\\
		&=(\mathcal{T}_1, Y_1\sqcup Y_2\cup Y_3),
		\end{align*}
whenever $X_3=X_2\setminus Y_2$ and $\mathcal{T}_3=\mathcal{T}_2\restr X_3$, as well as $X_2=X_1\setminus Y_1$ and $\mathcal{T}_2=\mathcal{T}_1\restr {X_2}$. Then
		\[(\mathcal{T}_1,Y_1)\ast \big( (\mathcal{T}_2,Y_2)\ast (\mathcal{T}_3,Y_3) \big) = \begin{cases}(\mathcal{T}_1,Y_1\sqcup Y_2\sqcup Y_3) & \\
		&\hskip-15mm \text{if $X_2=X_1\setminus Y_1$, $X_3=X_2\setminus Y_2$, $\mathcal T_2=\mathcal T_1\restr{X_2}$ and $\mathcal T_3=\mathcal T_2\restr{X_3}$}\\
		0 & \text{if not.} \end{cases}\]
		We therefore conclude that for all $(\mathcal{T}_1,Y_1), (\mathcal{T}_2,Y_2), (\mathcal{T}_3,Y_3) \in \mathbb{D}_X$, we have
		$$\big( (\mathcal{T}_1,Y_1)\ast (\mathcal{T}_2,Y_2) \big) \ast (\mathcal{T}_3,Y_3)=(\mathcal{T}_1,Y_1)\ast \big( (\mathcal{T}_2,Y_2)\ast (\mathcal{T}_3,Y_3) \big),$$
		which proves the associativity of the product $\ast $.
	\end{proof}
	\subsection{Associative product on $\widetilde{\mathcal{D}}$}
	Recall here that an element $(\mathcal{T}, \mathcal{T}^{\prime})$ belongs to $\widetilde{\mathbb{D}}_X$ if $\mathcal{T}$ and $\mathcal{T}^{\prime}$ are both topologies on $X$ such that $\mathcal{T}^{\prime}\oprec \mathcal{T}$.
	\begin {theorem} 
	The product $\divideontimes: \widetilde{\mathbb{D}}_X \otimes \widetilde{\mathbb{D}}_X \longrightarrow \widetilde{\mathbb{D}}_X$, defined by:
	\[(\mathcal{T}_1,\mathcal{T}^{\prime}_1)\divideontimes (\mathcal{T}_2,\mathcal{T}^{\prime}_2) \longmapsto \begin{cases}(\mathcal{T}_1,\mathcal{U}) & \text{if $\mathcal{T}_2=\mathcal{T}_1/ \mathcal{T}^{\prime}_1$,}\\
	0 & \text{if not} \end{cases}\]
	is associative, where $\mathcal{U}$ is defined by $\mathcal{T}^{\prime}_2=\mathcal{U}/\mathcal{T}^{\prime}_1$. (\cite[Proposition 2.7]{acg10}, see Lemma \ref{lemme-sandwich}).
\end{theorem} 
\begin{proof}
	Let $(\mathcal{T}_1,\mathcal{T}^{\prime}_1), (\mathcal{T}_2,\mathcal{T}^{\prime}_2)$ and $(\mathcal{T}_3,\mathcal{T}^{\prime}_3)$ be three elements of $\widetilde{\mathbb{D}}_X$, i.e., $\mathcal{T}^{\prime}_1\oprec \mathcal{T}_1$, $\mathcal{T}^{\prime}_2\oprec \mathcal{T}_2$ and $\mathcal{T}^{\prime}_3\oprec \mathcal{T}_3$.
	We suppose first that $\mathcal{T}_2=\mathcal{T}_{1}/\mathcal{T}^{\prime}_1$, otherwise the result is zero.
	\begin{align*}
	\big( (\mathcal{T}_1,\mathcal{T}^{\prime}_1)\divideontimes (\mathcal{T}_2,\mathcal{T}^{\prime}_2) \big) \divideontimes (\mathcal{T}_3,\mathcal{T}^{\prime}_3)&=(\mathcal{T}_1,\mathcal{U})\divideontimes (\mathcal{T}_3,\mathcal{T}^{\prime}_3),
	\end{align*}
	 where $\mathcal{U}$ is defined by $\mathcal{T}^{\prime}_2=\mathcal{U}/\mathcal{T}^{\prime}_1$, then\\
	 \begin{align*}
	\big( (\mathcal{T}_1,\mathcal{T}^{\prime}_1)\divideontimes (\mathcal{T}_2,\mathcal{T}^{\prime}_2) \big) \divideontimes (\mathcal{T}_3,\mathcal{T}^{\prime}_3)&=(\mathcal{T}_1, \mathcal{V})
	\end{align*}
	 where $\mathcal{V}$ and $\mathcal{U}$ are defined by $\mathcal{T}^{\prime}_3=\mathcal{V}/\mathcal{U}$, and $\mathcal{T}^{\prime}_2=\mathcal{U}/\mathcal{T}^{\prime}_1$. Then,\\
	\[\big( (\mathcal{T}_1,\mathcal{T}^{\prime}_1)\divideontimes (\mathcal{T}_2,\mathcal{T}^{\prime}_2) \big) \divideontimes (\mathcal{T}_3,\mathcal{T}^{\prime}_3) \longmapsto \begin{cases}(\mathcal{T}_1,\mathcal{V}) & \text{if $\mathcal{T}_2=\mathcal{T}_1/\mathcal{T}^{\prime}_1$ and $\mathcal{T}_3=\mathcal{T}_1/\mathcal{U}$} \\
	0 & \text{if not.} \end{cases}\]
	Where $\mathcal{T}^{\prime}_2=\mathcal{U}/\mathcal{T}^{\prime}_1$, and $\mathcal{T}^{\prime}_3=\mathcal{V}/\mathcal{U}$.\\ 
	
	On the other hand, for $\mathcal{T}_3=\mathcal{T}_2/\mathcal{T}^{\prime}_2$, we have
	\begin{align*}
	(\mathcal{T}_1,\mathcal{T}^{\prime}_1)\divideontimes \big( (\mathcal{T}_2,\mathcal{T}^{\prime}_2)\divideontimes (\mathcal{T}_3,\mathcal{T}^{\prime}_3)\big) &=(\mathcal{T}_1,\mathcal{T}^{\prime}_1)\divideontimes (\mathcal{T}_2, \mathcal{W}),
	\end{align*}
	where $\mathcal{T}^{\prime}_3=\mathcal{W}/\mathcal{T}^{\prime}_2$, where $\mathcal{T}_3=\mathcal{T}_2/\mathcal{T}^{\prime}_3$, and $\mathcal{T}_2=\mathcal{T}_1/\mathcal{T}^{\prime}_2$. Then
	$\mathcal{T}_3=\mathcal{T}_2/\mathcal{T}^{\prime}_3=\mathcal{T}_1/\mathcal{W}$, and $\mathcal{W}=\mathcal{Z}/\mathcal{T}^{\prime}_1$. Hence
	\[(\mathcal{T}_1,\mathcal{T}^{\prime}_1)\divideontimes \big((\mathcal{T}_2,\mathcal{T}^{\prime}_2)\divideontimes (\mathcal{T}_3,\mathcal{T}^{\prime}_3)\big) = \begin{cases}(\mathcal{T}_1,\mathcal{Z}) & \text{if $\mathcal{T}_3=\mathcal{T}_1/(\mathcal{T}^{\prime}_2\sqcup \mathcal{T}^{\prime}_3)$ and $\mathcal{T}_2=\mathcal{T}_1/\mathcal{T}^{\prime}_2$} \\
	0 & \text{if not.} \end{cases}\]
	Where $\mathcal{Z}$ is defined by $\mathcal{W}=\mathcal{Z}/\mathcal{T}^{\prime}_1$, and $\mathcal{T}^{\prime}_3=\mathcal{W}/\mathcal{T}^{\prime}_2$. It remains to show that $\mathcal{V}=\mathcal{Z}$: We have
	$$\mathcal{V}/\mathcal{U}=\mathcal{T}^{\prime}_3=\mathcal{W}/\mathcal{T}^{\prime}_2=(\mathcal{Z}/\mathcal{T}^{\prime}_1)/(\mathcal{U}/\mathcal{T}^{\prime}_1)=\mathcal{Z}/\mathcal{U}.$$
	Moreover, $\mathcal{T}_3=\mathcal{T}_2/\mathcal{T}^{\prime}_2=(\mathcal{T}_1/\mathcal{T}^{\prime}_1)/(\mathcal{U}/\mathcal{T}^{\prime}_1)=\mathcal{T}_1/\mathcal{U}$. We therefore conclude that for all $(\mathcal{T}_1,\mathcal{T}^{\prime}_1)$, $(\mathcal{T}_2,\mathcal{T}^{\prime}_2)$ and $(\mathcal{T}_3,\mathcal{T}^{\prime}_3)$ in $\widetilde{\mathbb{D}}_X$, we have
	$$\big( (\mathcal{T}_1,\mathcal{T}^{\prime}_1)\divideontimes (\mathcal{T}_2,\mathcal{T}^{\prime}_2)\big) \divideontimes (\mathcal{T}_3,\mathcal{T}^{\prime}_3)=(\mathcal{T}_1,\mathcal{T}^{\prime}_1)\divideontimes \big((\mathcal{T}_2,\mathcal{T}^{\prime}_2)\divideontimes (\mathcal{T}_3,\mathcal{T}^{\prime}_3)\big),$$
	which proves the associativity of the product $\divideontimes $.
\end{proof}
\section{Other relations between both doubling species $\mathbb{D}$ and $\widetilde{\mathbb{D}}$}\label{sect:comp}
\subsection{An action of $\widetilde{\mathbb D}$ on $\mathbb D$}
\begin{definition}
For any finite set $X$, let
	 $\Psi : \widetilde{\mathbb{D}}_X \otimes \mathbb{D}_X \longrightarrow \mathbb{D}_X$ be the map defined by:
	\[\Psi \big( (\mathcal{T},\mathcal{T}^{\prime})\otimes (\mathcal{U},Y)\big)  = \begin{cases}(\mathcal{T}, Y) & \text{if $\mathcal{U}=\mathcal{T}/\mathcal{T}^{\prime}$ } \\
	0 & \text{if not.} \end{cases}\]
\end{definition}
\begin{proposition}
	$\Psi $  is an action of $\widetilde{\mathbb{D}}$ on $\mathbb{D}$.
\end{proposition}
\begin{proof}
	We have to verify the commutativity of this diagram:
	$$
	\xymatrix{
		\widetilde{\mathbb{D}}_X \otimes \widetilde{\mathbb{D}}_X \otimes \mathbb{D}_X \ar[rr]^{id\otimes \Psi } \ar[d]_{\divideontimes \otimes id } && \widetilde{\mathbb{D}}_X\otimes \mathbb{D}_X\ \ar[d]^{\Psi }\\
		\widetilde{\mathbb{D}}_X\otimes \mathbb{D}_X \ar[rr]_{\Psi } && \mathbb{D}_X
	}
	$$
	Let $(\mathcal{T}_1,\mathcal{T}^{\prime}_1)$ and $(\mathcal{T}_2,\mathcal{T}^{\prime}_2)$ be two elements of $\widetilde{\mathbb{D}}_X$ , and $(\mathcal{U},Y)\in \mathbb{D}_X$. We have
	\[(id\otimes \Psi)\big( (\mathcal{T}_1,\mathcal{T}^{\prime}_1)\otimes (\mathcal{T}_2,\mathcal{T}^{\prime}_2)\otimes (\mathcal{U},Y)\big)  = \begin{cases}(\mathcal{T}_1, \mathcal{T}^{\prime}_1)\otimes (\mathcal{T}_2,Y) & \text{if $\mathcal{U}=\mathcal{T}_2/\mathcal{T}^{\prime}_2$ } \\
	0 & \text{if not.} \end{cases}\]
	Then,
	\[\Psi \circ (id\otimes \Psi)\big( (\mathcal{T}_1,\mathcal{T}^{\prime}_1)\otimes (\mathcal{T}_2,\mathcal{T}^{\prime}_2)\otimes (\mathcal{U},Y)\big)  = \begin{cases}(\mathcal{T}_1, Y) & \text{if $\mathcal{U}=\mathcal{T}_2/\mathcal{T}^{\prime}_2$ and $\mathcal{T}_2=\mathcal{T}_1/\mathcal{T}^{\prime}_1$ } \\
	0 & \text{if not.} \end{cases}\]
	On the other hand, we have
	\[(\divideontimes \otimes id)\big( (\mathcal{T}_1,\mathcal{T}^{\prime}_1)\otimes (\mathcal{T}_2,\mathcal{T}^{\prime}_2)\otimes (\mathcal{U},Y)\big)  = \begin{cases}(\mathcal{T}_1, \mathcal{V})\otimes (\mathcal{U},Y) & \text{if $\mathcal{T}_2=\mathcal{T}_1/\mathcal{T}^{\prime}_1$ } \\
	0 & \text{if not.} \end{cases}\]
	where $\mathcal{T}^{\prime}_2=\mathcal{V}/\mathcal{T}^{\prime}_1 $. Then,
	\[\Psi \circ (\divideontimes \otimes id)\big( (\mathcal{T}_1,\mathcal{T}^{\prime}_1)\otimes (\mathcal{T}_2,\mathcal{T}^{\prime}_2)\otimes (\mathcal{U},Y)\big)  = \begin{cases}(\mathcal{T}_1, Y) & \text{if $\mathcal{T}_2=\mathcal{T}_1/\mathcal{T}^{\prime}_1$ and $\mathcal{U}=\mathcal{T}_1/\mathcal{V}$ } \\
	0 & \text{if not.} \end{cases}\]
	Moreover, $\mathcal{U}=\mathcal{T}_2/\mathcal{T}^{\prime}_2=(\mathcal{T}_1/\mathcal{T}^{\prime}_1)/(\mathcal{V}/\mathcal{T}^{\prime}_1)=\mathcal{T}_1/\mathcal{V}$. We conclude then
	$$\Psi \circ (\divideontimes \otimes id)=\Psi \circ (id\otimes \Psi).$$
	which proves that $\Psi$ is an action of $\widetilde{\mathbb{D}}$ on $\mathbb{D}$.
\end{proof}
\subsection{The cointeraction}
	\begin{theorem}\label{cointeraction}
	For any finite set $X$, let
	 $\xi : \widetilde{\mathbb{D}}_X \otimes (\mathbb{D}\otimes \mathbb{D})_X \longrightarrow \widetilde{\mathbb{D}}_X \otimes (\mathbb{D}\otimes \mathbb{D})_X$ be the map defined by:
	 \begin{equation}\label{eq:xi}
	 \xi\big((\mathcal{T},\widetilde{\mathcal{T}})\otimes (\mathcal{T}_1,Z)\otimes (\mathcal{T}_2,W)  \big)=(\mathcal{T}\restr{Z}\mathcal{T}\restr{X\backslash Z},\widetilde{\mathcal{T}}\restr{Z}\widetilde{\mathcal{T}}\restr{X\backslash Z})\otimes (\mathcal{T}_1,Z)\otimes (\mathcal{T}_2,W).
	 \end{equation}
	The following diagram is commutative:
$$
\xymatrix{
\mathbb{D}_X \ar[rr]^\Phi \ar[d]_{\Delta} && \widetilde{\mathbb{D}}_X \otimes \mathbb{D}_X \ar[d]^{Id \otimes \Delta}\\
	(\mathbb{D}\otimes\mathbb{D})_X \ar[d]_{\Phi \otimes \Phi } && \widetilde{\mathbb{D}}_X \otimes (\mathbb{D} \otimes \mathbb{D})_X \ar[d]^{\xi}\\
\bigoplus \limits_{Y\subset X}\widetilde{\mathbb{D}}_Y \otimes \mathbb{D}_Y \otimes \widetilde{\mathbb{D}}_{X\backslash Y} \otimes \mathbb{D}_{X\backslash Y} \ar[rr]_{ m^{1,3} } && \widetilde{\mathbb{D}}_X \otimes (\mathbb{D} \otimes \mathbb{D})_X
}
$$
i.e.,
$$\xi \circ(id\otimes \Delta )\circ \Phi =(m^{1,3}) \circ(\Phi \otimes \Phi )\circ \Delta.$$
\end{theorem} 
\begin{proof}
In \eqref{eq:xi}, the finite set $X$ is partitioned into two subsets $X_1$ and $X_2$, and $Z$ (resp. $W$) is an open subset of $X_1$ (resp. $X_2$) for $\mathcal T_1$ (resp. $\mathcal T_2$). According to Proposition \ref{restriction}, the relation $\mathcal{T}\restr{Z}\mathcal{T}\restr{X\backslash Z}\oprec\mathcal{T}^{\prime}\restr{Z}\mathcal{T}^{\prime}\restr{X\backslash Z}$ holds, hence the map $\xi$ is well defined. For any $(\mathcal{T},Y)\in \mathbb{D}_X$, we have
	\begin{align*}
	(id\otimes \Delta )\circ \Phi (\mathcal{T},Y)&=(id\otimes \Delta )\left(\sum \limits_{\underset{\scriptstyle\widetilde{\mathcal{T}}\srestr{X\backslash Y}=D_{X\backslash Y,\widetilde{\mathcal{T}}} }{\widetilde{\mathcal{T}}\soprec  \mathcal{T},\,Y\in \mathcal{T}/ \widetilde{\mathcal{T}}}}(\mathcal{T},\widetilde{\mathcal{T}})\otimes (\mathcal{T}/ \widetilde{\mathcal{T}},Y)\right)\\
	&=\sum \limits_{\underset{\scriptstyle\widetilde{\mathcal{T}}\srestr{X\backslash Y}=D_{X\backslash Y,\widetilde{\mathcal{T}}} }{\widetilde{\mathcal{T}}\soprec  \mathcal{T},\hskip1mm Y\in \mathcal{T}/ \widetilde{\mathcal{T}}}}\ \sum\limits_{\underset{}{Z\in (\mathcal{T}/ \widetilde{\mathcal{T}})\srestr{Y}}}(\mathcal{T},\widetilde{\mathcal{T}})\otimes \big((\mathcal{T}/ \widetilde{\mathcal{T}})\restr{Z},Z\big)\otimes \big((\mathcal{T}/ \widetilde{\mathcal{T}})\restr{X\backslash Z},Y\backslash Z \big)
	\end{align*}
	therefore
		\begin{eqnarray*}
	\xi\circ (id\otimes \Delta )\circ \Phi (\mathcal{T},Y)=\sum \limits_{\underset{\scriptstyle\widetilde{\mathcal{T}}\srestr{X\backslash Y}=D_{X\backslash Y,\widetilde{\mathcal{T}}} }{\widetilde{\mathcal{T}}\soprec  \mathcal{T},\hskip1mm Y\in \mathcal{T}/ \widetilde{\mathcal{T}}}}\ \sum\limits_{\underset{}{Z\in (\mathcal{T}/ \widetilde{\mathcal{T}})\srestr{Y}}}(\mathcal{T}\restr{Z}\mathcal{T}\restr{X\backslash Z},\widetilde{\mathcal{T}}\restr{Z}\widetilde{\mathcal{T}}\restr{X\backslash Z})\otimes \big((\mathcal{T}/ \widetilde{\mathcal{T}})\restr{Z},Z\big)\otimes \big((\mathcal{T}/ \widetilde{\mathcal{T}})\restr{X\backslash Z},Y\backslash Z \big).
		\end{eqnarray*}
On the other hand, we have
\begin{eqnarray*}
(\Phi \otimes \Phi )\circ \Delta (\mathcal{T},Y)\\
	&\hskip -15mm =\displaystyle\sum_{Z\in \mathcal{T}\srestr{Y}}\ \sum_{\mathcal{T}^{\prime}\soprec \mathcal{T}\srestr{Z}}\ \sum \limits_{\underset{ \scriptstyle \mathcal{T}^{\prime \prime}\srestr{X\backslash Y}=D_{X\backslash Y,\mathcal{T}^{\prime \prime}}}{\mathcal{T}^{\prime \prime}\soprec \mathcal{T}\srestr{X\backslash Z},\hskip1mmY\backslash Z\in \mathcal{T}\srestr{X\backslash Z}/ \mathcal{T}^{\prime \prime }}}(\mathcal{T}\restr{Z},\mathcal{T}^{\prime})\otimes (\mathcal{T}\restr{Z}/ \mathcal{T}^{\prime}, Z)\otimes (\mathcal{T}\restr{X\backslash Z}, \mathcal{T}^{\prime \prime})\otimes (\mathcal{T}\restr{X\backslash Z}/\mathcal{T}^{\prime \prime}, Y\backslash Z),
\end{eqnarray*}
therefore
\begin{eqnarray*}
	m^{1,3}\circ (\Phi \otimes \Phi )\circ \Delta (\mathcal{T},Y)\\
	&\hskip -10mm =\displaystyle\sum_{Z\in \mathcal{T}\srestr{Y}}\ \sum_{\mathcal{T}^{\prime}\soprec \mathcal{T}\srestr{Z}}\ \sum \limits_{\underset{ \scriptstyle \mathcal{T}^{\prime \prime}\srestr{X\backslash Y}=D_{X\backslash Y,\mathcal{T}^{\prime \prime}}}{\mathcal{T}^{\prime \prime}\soprec \mathcal{T}\srestr{X\backslash Z},\hskip1mmY\backslash Z\in \mathcal{T}\srestr{X\backslash Z}/ \mathcal{T}^{\prime \prime }}}(\mathcal{T}\restr{Z}\mathcal{T}\restr{X\backslash Z},\mathcal{T}^{\prime}\mathcal{T}^{\prime \prime})\otimes (\mathcal{T}/ \mathcal{T}^{\prime}, Z)\otimes (\mathcal{T}\restr{X\backslash Z}/\mathcal{T}^{\prime \prime}, Y\backslash Z).
\end{eqnarray*}
	\ignore{The map $(Z, \mathcal{T}^{\prime }, \mathcal{T}^{\prime \prime} )\mapsto (Z, \mathcal{T}^{\prime }\mathcal{T}^{\prime \prime})$ is
	a bijection from the set of triples $(Z, \mathcal{T}^{\prime }, \mathcal{T}^{\prime \prime} )$ with $ Z\in \mathcal{T}\restr{Y}$ and $\mathcal{T}^{\prime \prime}\oprec \mathcal{T}\restr{X\backslash Z},  \mathcal{T}^{\prime \prime}\restr{X\backslash Y}=D_{X\backslash Y,\mathcal{T}^{\prime \prime}},Y\backslash Z\in \mathcal{T}\restr{X\backslash Z}/ \mathcal{T}^{\prime \prime },\mathcal{T}^{\prime}\oprec \mathcal{T}\restr{Z}$
	, onto
	the set of pairs $(Z,\widetilde{\mathcal{T}})$ with $\widetilde{\mathcal{T}}\oprec \mathcal{T},  \widetilde{\mathcal{T}}\restr{X\backslash Y}=D_{X\backslash Y,\widetilde{\mathcal{T}}}$ and $Z\in (\mathcal{T}/ \widetilde{\mathcal{T}})\restr{Y} $. The inverse map is given by $(Z,\widetilde{\mathcal{T}})\mapsto (Z, \widetilde{\mathcal{T}}\restr{Z},\widetilde{\mathcal{T}}\restr{X\backslash Z})$. }To show that the both expressions above coincide, we use the two following lemmas.
\begin{lemma}\label{lem:equivalence}
		Let $\mathcal{T}$ and $\widetilde{\mathcal{T}}$ be two topologies on $X$, such that $\widetilde{\mathcal{T}} \oprec \mathcal{T}$ and let $Y\in \mathcal{T}$. Then $Y\in \mathcal{T}/ \widetilde{\mathcal{T}}$ if and only if both $Y$ and $X\backslash Y$ are open subsets of $X$ for $\widetilde{\mathcal{T}}$.
	\end{lemma}
\begin{proof}
From $Y\in \mathcal{T}$ and $\widetilde{\mathcal{T}} \oprec \mathcal{T}$ we immediately get $Y\in \widetilde{\mathcal{T}}$. Now let $x\in X\backslash Y$ and $y\in X$ with $x\leq_{\widetilde{\mathcal{T}}} y$. From $x\leq_{\widetilde{\mathcal{T}}} y$  we get $y\leq_{\mathcal{T}/ \widetilde{\mathcal{T}}} x$. Suppose that $y\in Y$ then $x\in Y$ (because $Y\in \mathcal{T}/ \widetilde{\mathcal{T}}$), which is absurd. Hence $X\backslash Y\in \widetilde{\mathcal{T}}$.\\

Conversely, suppose that both $Y$ and $X\backslash Y$ are open subsets of $X$ for $\widetilde{\mathcal{T}}$, let $y\in Y$ and let $z\in X$ with $y\le_{\mathcal T/\widetilde{\mathcal{T}}}z$. There is a chain
$$yRt_1\cdots Rt_kRz$$
with $t_1,\ldots,t_k\in X$, where $aRb$ means $a\le_{\mathcal T}b$ or $a\ge_{\widetilde{\mathcal T}}b$. Supposing $a\in Y$ we have either $a\le_{\mathcal T}b$ which yields $b\in Y$, or $b\le_{\widetilde{\mathcal T}}a$, which would yield the contradiction $a\in X\backslash Y$ if $b$ were to belong to $X\backslash Y$. Hence we necessarily have $b\in Y$. Progressing along the chain, from $y\in Y$ we therefore infer $z\in Y$.
\end{proof}
\begin{lemma}\label{T-split}
Let $\mathcal T$ be a topology on a finite set $X$, and let $Z$ be an open subset of $X$ for $T$. Let $\mathcal U$ be the topology $\mathcal T\restr{Z}\mathcal T\restr{X\backslash Z}$. Then we have
$$\mathcal U\oprec\mathcal T.$$
\end{lemma}
\begin{proof}
The relation $\mathcal U\prec\mathcal T$ is obvious. Any connected component $W$ for $\mathcal U$ is contained either in $Z$ or in $X\backslash Z$, hence $\mathcal T\restr W=\mathcal U\restr W$. Finally, consider $x,y\in X$ such that $x\sim_{\mathcal T/\mathcal U} y$. There is a chain $xRt_1R\cdots Rt_kRx$ with some $j\in\{1,\ldots,k\}$ such that $t_j=y$. By the argument which was used in the end of the proof of Proposition \ref{restriction}, the whole chain belongs to the same $\mathcal U$-connected component, hence we have $x\widetilde Rt_1\widetilde R\cdots \widetilde Rt_k\widetilde Rx$, where $a\widetilde Rb$ stands for $a\le_{\mathcal U} b$ or $b\le_{\mathcal U} a$. Hence $x\sim_{\mathcal U/\mathcal U}y$.
\end{proof}
\noindent {\it Proof of theorem \ref{cointeraction} (continued).}
	Let $\mathcal E$ be the set of triples $(Z,\mathcal T^{\prime},\mathcal T^{\prime\prime})$ which occur in the expression of $m^{1,3}\circ (\Phi \otimes \Phi )\circ \Delta (\mathcal{T},Y)$ above, i.e. subject to the conditions
$$Z\in \mathcal{T}\restr{Y}, \hskip 5mm \mathcal{T}^{\prime \prime}\oprec \mathcal{T}\restr{X\backslash Z},\hskip 5mm \mathcal{T}^{\prime}\oprec \mathcal{T}\restr{Z}, \hskip 5mm  
	\mathcal{T}^{\prime \prime}\restr{X\backslash Y}=D_{X\backslash Y,\mathcal{T}^{\prime \prime}},\hskip 5mm Y\backslash Z\in \mathcal{T}\restr{X\backslash Z}/ \mathcal{T}^{\prime \prime },$$
	and let $\mathcal F$ be the set of pairs $(Z,\widetilde{\mathcal T})$ which occur in the expression of $\xi\circ (id\otimes \Delta )\circ \Phi (\mathcal{T},Y)$ above, i.e. subject to the conditions
$$
\widetilde{\mathcal T}\oprec\mathcal T,\hskip 3mm Y\in\mathcal T/\widetilde{\mathcal T},\hskip 3mm
\widetilde{\mathcal T}\restr{X\backslash Y}=D_{X\backslash Y,\widetilde{\mathcal T}},\hskip 3mm Z\in(\mathcal T/\widetilde{\mathcal T})\restr Y.
$$

To prove Theorem \ref{cointeraction}, it suffices to show that $(Z,\mathcal T^{\prime},\mathcal T^{\prime\prime})\mapsto (Z,\mathcal T^{\prime}\mathcal T^{\prime\prime})$ is a bijection from $\mathcal E$ onto $\mathcal F$. For any $(Z,\mathcal T^{\prime}\mathcal T^{\prime\prime})\in\mathcal E$, it is clear from Lemma \ref{T-split} that $\widetilde{\mathcal T}\oprec\mathcal T$ holds, with $\widetilde{\mathcal T}:=\mathcal T^{\prime}\mathcal T^{\prime\prime}$. From $Y\in\mathcal T$ we get $Y\backslash Z\in\mathcal T\restr{X\backslash Z}$. Together with $Y\backslash Z\in\mathcal T\restr{X\backslash Z}/\mathcal T^{\prime\prime}$ one deduces from Lemma \ref{lem:equivalence} that both $Y\backslash Z$ and $X\backslash Y$ are open subsets of $X\backslash Z$ for $\mathcal T^{\prime\prime}$. Hence we have a partition
\begin{equation}\label{split-three}
X=Z\sqcup Y\backslash Z\sqcup X\backslash Y
\end{equation}
of $X$ into three open subsets for the topology $\widetilde {\mathcal T}$. From $Y\in \widetilde{\mathcal T}$ and $X\backslash Y\in\widetilde{\mathcal T}$ we get $Y\in \mathcal T/\widetilde{\mathcal T}$ from Lemma \ref{lem:equivalence}. We also have
$$\widetilde{\mathcal T}\restr{X\backslash Y}={\mathcal T}^{\prime\prime}\restr{X\backslash Y}
=D_{X\backslash Y,{\mathcal T}^{\prime\prime}}=D_{X\backslash Y,\widetilde{\mathcal T}}.$$
Finally, from the fact that both $Z$ and $X\backslash Z$ are open for $\widetilde{\mathcal T}$ and $Z\in T$, we get $Z\in \mathcal T/\widetilde{\mathcal T}$ from Lemma \ref{lem:equivalence}, hence $Z\in (\mathcal T/\widetilde{\mathcal T})\restr Y$. This proves $(Z,\widetilde{\mathcal T})\in\mathcal F$.\\

Conversely, for any $(Z,\widetilde{\mathcal T})\in\mathcal F$, from $Y\in \mathcal T/\widetilde{\mathcal T}$ and Lemma \ref{lem:equivalence} we get that both $Y$ and $X\backslash Y$ are open subsets of $X$ for $\widetilde{\mathcal T}$, and from $Z\in\mathcal T/\widetilde{\mathcal T}$ and the same lemma we get that both $Z$ and $X\backslash Z$ are open subsets of $X$ for $\widetilde{\mathcal T}$. Hence the splitting \eqref{split-three} into three open subsets holds, and we have $\widetilde{\mathcal T}=\mathcal T^{\prime}\mathcal T^{\prime\prime}$, with $\mathcal T^{\prime}:=\widetilde{\mathcal T}\restr Z$ and $\mathcal T^{\prime\prime}:=\widetilde{\mathcal T}\restr {X\backslash Z}$. The triple $(Y,\mathcal T^{\prime},\mathcal T^{\prime\prime})$ verifies the five required conditions to belong to $\mathcal E$. Both correspondences from $\mathcal E$ to $\mathcal F$ are obvioulsly inverse to each other, which ends up the proof of Theorem \ref{cointeraction}.

\end{proof}

	\begin{remark}\rm
If we apply the functor $\overline{\mathcal{K}}$, we notice here that this diagram yields the commutative diagram:

$$
\xymatrix{
\mathcal{D} \ar[rr]^\Phi \ar[d]_{\Delta} && \widetilde{\mathcal{D}} \otimes \mathcal{D} \ar[d]^{Id \otimes \Delta}\\
	\mathcal{D}\otimes\mathcal{D} \ar[d]_{\Phi \otimes \Phi } && \widetilde{\mathcal{D}} \otimes \mathcal{D} \otimes \mathcal{D} \ar[d]^{\xi}\\
\widetilde{\mathcal{D}} \otimes \mathcal{D} \otimes \widetilde{\mathcal{D}} \otimes \mathcal{D} \ar[rr]_{ m^{1,3} } &&\widetilde{\mathcal{D}} \otimes \mathcal{D} \otimes \mathcal{D}
}
$$
where $\Phi$ is a shorthand for $\overline{\mathcal{K}}(\Phi)$ and so on.
	\end{remark}


\end{document}